\documentclass[a4paper,12pt,reqno]{amsart}
\usepackage{amsfonts}
\usepackage{amsmath}
\usepackage{amssymb}
\usepackage[a4paper]{geometry}
\usepackage{mathrsfs}
\usepackage{csquotes}
\usepackage{enumitem}
\usepackage{dirtytalk}

\usepackage[symbol]{footmisc}

\usepackage[colorlinks]{hyperref}
\renewcommand\eqref[1]{(\ref{#1})} 
\numberwithin{equation}{section}
\theoremstyle{plain}
\newtheorem{thm}{Theorem}[section]
\newtheorem{prop}[thm]{Proposition}
\newtheorem{cor}[thm]{Corollary}
\newtheorem{lem}[thm]{Lemma}
 
\theoremstyle{definition}

\newtheorem{rem}[thm]{Remark}

\newcommand{\G}{\mathbb G}
\newcommand{\X}{\mathbb X}

\def\e[#1]{{\textrm{e}}^{#1}}

\def\G{{\mathbb G}}

\begin{document}

   \title[Hardy-Sobolev-Rellich, HLS and CKN inequalities on Lie groups]
 {Hardy-Sobolev-Rellich, Hardy-Littlewood-Sobolev and Caffarelli-Kohn-Nirenberg inequalities on general Lie groups}

\author[M. Ruzhansky]{Michael Ruzhansky}
\address{
  Michael Ruzhansky:
 \endgraf
    Department of Mathematics: Analysis, Logic and Discrete Mathematics
  \endgraf
    Ghent University, Belgium
   \endgraf
  and
  \endgraf
  School of Mathematical Sciences
    \endgraf
    Queen Mary University of London, United Kingdom
   \endgraf
  {\it E-mail address} {\rm
michael.ruzhansky@ugent.be}
  }

\author[N. Yessirkegenov]{Nurgissa Yessirkegenov}
\address{
  Nurgissa Yessirkegenov:
  \endgraf
  SDU University, Kaskelen, Kazakhstan
  \endgraf
  and 
  \endgraf
  Institute of Mathematics and Mathematical Modeling, Kazakhstan
    \endgraf
  {\it E-mail address} {\rm nurgissa.yessirkegenov@gmail.com}
  }
 \dedicatory{Dedicated to Luigi Rodino on the occasion of his 75$^{th}$ birthday}

\thanks{This research is funded by the Committee of Science of the Ministry of Science and Higher Education of the Republic of Kazakhstan (Grant No. AP14871691) and by the FWO Odysseus 1 grant G.0H94.18N: Analysis and Partial Differential Equations. MR was also supported by the EPSRC grant EP/R003025/2 and by the Methusalem programme of the Ghent University Special
Research Fund (BOF) (Grant number 01M01021).}


     \begin{abstract} In this paper we establish a number of geometrical
inequalities such as Hardy, Sobolev, Rellich, Hardy-Littlewood-Sobolev, Caffarelli-Kohn-Nirenberg, Gagliardo-Nirenberg inequalities and their critical versions for an ample
class of sub-elliptic differential operators on general connected Lie groups, which include both unimodular and non-unimodular cases in compact and noncompact settings. We also obtain the corresponding uncertainty type principles.

     \end{abstract}
     \maketitle

\section{Introduction}
\label{SEC:intro}
Let $\G$ be a connected Lie group, let $\rho$ and $\lambda$ denote a right and left Haar measure on $\G$, $\delta$ the modular function on $\G$ so that $d \lambda(x)=\delta(x) d \rho(x)$, $x \in \mathbb{G}$. Let $X=\left(X_{1}, \ldots, X_{n} \right)$ be a collection of left-invariant vector fields on $\G$ satisfying H\"{o}rmander's condition.

Let $\chi$ denote a positive continuous character and $e$ the identity element on $\mathbb{G}$, and define $\mu_{\chi}=\chi \rho$ and let $\Delta_{\chi}$ be the differential operator
\begin{equation}\label{Lap_chi}
-\sum_{j=1}^{n}(X_{j}^{2}+X_{j}(\chi)(e) X_{j}).
\end{equation}

It was shown in \cite{HMM05} that $\Delta_\chi$, when initially defined on $C_0^{\infty}(\mathbb{G})$, is essentially self-adjoint in $L^2\left(\mu_\chi\right)$ and conversely, if $\mu$ is a positive Borel measure on $\mathbb{G}$ such that $\Delta_1-X$ is essentially self-adjoint in $L^2(\mu)$, where $X$ is a left-invariant vector field, then there exists a positive continuous character $\chi$ on $\mathbb{G}$ such that $X=\Delta_1-\Delta_\chi$ and $\mu=\mu_\chi$. By a common abuse of notation, from now on we will denote the unique self-adjoint realization of $\Delta_\chi$ by the same symbol. Thus, the family of sub-Laplacians with drift $\Delta_{\chi}$ turns out to be \say{the} natural family of second order differential operators for which it is reasonable to apply functional calculus results and methods to define and study function spaces, and regularity of solutions of differential equations. In particular, when $\chi=\delta$, so that $\mu_{\chi}=\lambda$ is the left Haar measure, then $\Delta_{\delta}$ is the {\em intrinsic} sub-Laplacian, introduced by Agrachev, Boscain, Gauthier and Rossi \cite{ABGR09}.

Sobolev spaces defined in terms of $\Delta_{\chi}$ were introduced and studied in \cite{BPTV18}, where various embedding and algebra properties were proved. When $1<p<\infty$ and $\alpha\geq 0$, the Sobolev spaces $L_{\alpha}^{p}\left(\mu_{\chi}\right)$ were defined as the completion of $C_{0}^{\infty}(\mathbb{G})$ with respect to the norm
$$
\left\| f\right\|_{L_{\alpha}^{p}\left(\mu_{\chi}\right)}:=\left\| f\right\|_{L^{p}(\mu_{\chi})}+\left\| \Delta_{\chi}^{\alpha / 2} f\right\|_{L^{p}(\mu_{\chi})}.
$$
These spaces, in particular these measures, appear naturally in embeddings and algebra properties, where the case $\chi=\delta$, that is, $\mu_{\chi}=\lambda$, plays a fundamental role. In the case when $\G$ is a unimodular group and with $\chi=1$, we note that the spaces $L^{p}_{\alpha}(\mu_{\chi})$ coincide with the Sobolev spaces defined by $\Delta_{1}=-\sum_{j=1}^{n}X_{j}^{2}$ (see \cite{CRT01}). In the case $\chi\neq1$, we note that this operator $\Delta_{1}$ is not symmetric on $L^{2}(\mu_{\chi})$, so that a Sobolev space adapted to $\mu_{\chi}$ when $\chi\neq1$ cannot be defined by means of fractional powers of $\Delta_{1}$. For more details we refer to \cite{HMM05}, \cite{PV17} and \cite{BPTV18}.

It became natural to study geometric inequalities for the scale of Sobolev spaces $L_{\alpha}^{p}\left(\mu_{\chi}\right)$, at least in the cases $1<p<\infty$ and $\alpha>0$.

In the unimodular case and with $\chi=1$, for embedding
theorems for these Sobolev spaces we refer to \cite{Fol75} on stratified group, and to \cite{FR16} and \cite{FR17} on graded groups, as well as to \cite{RY18} for the weighted versions. On general homogeneous groups, we refer to \cite{RS17} and \cite{RSY18}.

In the non-unimodular case, we refer to \cite{Var88} for the first-order Sobolev spaces when $\chi$ is a power of $\delta$, and to \cite{BPTV18} for the higher order case. 

For algebra properties of the Sobolev spaces, we refer to \cite{CRT01} on unimodular groups, and refer to \cite{PV17} and to \cite{BPTV18} on non-unimodular groups.

In this paper, we obtain a number of (versions of) classical geometric inequalities on Sobolev spaces in a unified way, in the setting of general Lie groups, for an ample class of sub-elliptic differential operators.

As usual, in this paper $A\lesssim B$ means that there exists a positive constant $c$ such that $A\leq c B$. If $A\lesssim B$ and $B\lesssim A$, then we write $A\approx B$. In these notations, if the left and right-hand sides feature some functions $f$, the constant (using this notation) does not depend on $f$.

Let us begin with the following {\bf Hardy-Sobolev-Rellich} inequality on general connected Lie groups:
\begin{thm}\label{main1_thm} Let $\G$ be a connected Lie group. Let $e$ be the identity element of $\G$, and let $\chi$ be a positive character of $\G$. Let $|x|:=d_{C}(e,x)$ denote the Carnot-Carath\'{e}odory distance from $e$ to $x$. Let $d$ be the local dimension of $\G$ as recalled in \eqref{V_less1} and $\alpha>0$, $0\leq \beta<d$, $1<p,q<\infty$. Then we have 
\begin{equation}\label{main1_thm_for1}
\left\|\frac{f}{|x|^{\frac{\beta}{q}}}\right\|_{L^{q}(\mu_{\chi^{q/p}\delta^{1-q/p}})}\lesssim
\|f\|_{L^{p}_{\alpha}(\mu_{\chi})}
\end{equation}
for all $q\geq p$ such that $1/p-1/q\leq \alpha/d-\beta/(dq)$.
\end{thm}
\begin{rem}
Note that in the case $\alpha\geq d/p$ the condition $1/p-1/q\leq \alpha/d-\beta/(dq)$ automatically holds true since
$$\frac{\alpha}{d}-\frac{\beta}{dq}\geq \frac{1}{p}-\frac{\beta}{dq}>\frac{1}{p}-\frac{1}{q},$$
which means the inequality \eqref{main1_thm_for1} holds for all $q\geq p$ when $\alpha\geq d/p$.
\end{rem}
Note that Theorem \ref{main1_thm} when $\beta=0$ implies the {\bf Sobolev embedding} on general connected Lie group
\begin{equation}\label{Sob_emb_unimod1}
L^{p}_{\alpha}(\mu_{\chi})\hookrightarrow L^{q}(\mu_{\chi^{q/p}\delta^{1-q/p}}).
\end{equation}
The Sobolev embedding \eqref{Sob_emb_unimod1} was proved in \cite{BPTV18} in the noncompact case. We also refer to the very recent work \cite{BPV21} for the investigation of the behaviour of the Sobolev embedding constant on a general connected Lie group, endowed with a left Haar measure. On nilpotent Lie groups, we refer to \cite{RY17} and \cite{RTY20} for the best constants in Sobolev, Gagliardo-Nirenberg and their critical cases for general left invariant homogeneous hypoelliptic differential operators.

Furthermore, for $q=p$ and $\beta/q=\alpha$ the inequality \eqref{main1_thm_for1} gives the following inhomogeneous {\bf Hardy inequality} on general connected Lie groups:
\begin{equation}\label{lef_rig_Hardy1}
\left\|\frac{f}{|x|^{\alpha}}\right\|_{L^{p}(\mu_{\chi})}\lesssim \|f\|_{L^{p}_{\alpha}(\mu_{\chi})},
\end{equation}
where $0\leq \alpha<d/p$. In particular, for $\chi=1$ and $\chi=\delta$ (respectively, with $\mu_{1}=\rho$ and $\mu_{\delta}=\lambda$) this gives both right and left versions of Hardy inequalities, respectively. We note that the inequality \eqref{lef_rig_Hardy1} was obtained for $\chi=1$ on stratified (hence also, in particular, nilpotent and unimodular) Lie groups in \cite{CCR15}. Note that when $\alpha=2$ the inequality \eqref{lef_rig_Hardy1} yields the following inhomogeneous {\bf Rellich inequality} on general connected Lie groups:
\begin{equation}\label{lef_rig_Hardy1_rel}
\left\|\frac{f}{|x|^{2}}\right\|_{L^{p}(\mu_{\chi})}\lesssim \|f\|_{L^{p}_{2}(\mu_{\chi})}\approx\left\|f \right\|_{L^{p}(\mu_{\chi})}+\left\|\Delta_{\chi}f \right\|_{L^{p}(\mu_{\chi})}, \; 0\leq 2p<d.
\end{equation}
Since the inequality \eqref{main1_thm_for1} contains the Hardy, Rellich and Sobolev inequalities, we call this inequality Hardy-Sobolev-Rellich inequality. The remaining cases of the inequality \eqref{main1_thm_for1} can be thought of as Sobolev embeddings in weighted $L^{q}$-spaces.

Moreover, we establish the Hardy-Sobolev-Rellich inequality \eqref{main1_thm_for1} in the {\bf critical case $\beta=d$} that involves a
logarithimic factor in the weight:
\begin{thm}\label{main3_thm} Let $\G$ be a connected Lie group. Let $\chi$ be a positive character of $\G$. Let $1<p<r<\infty$ and $1/p+1/p'=1$.
Then we have
\begin{equation}\label{main3_thm_for1}
\left\|\frac{f}{\left(\log\left({\rm e}+\frac{1}{|x|}\right)\right)^{\frac{r}{q}}
|x|^{\frac{d}{q}}}\right\|_{L^{q}(\mu_{\chi^{q/p}\delta^{1-q/p}})}\lesssim \|f\|_{L^{p}_{d/p}(\mu_{\chi})}
\end{equation}
for every $q\in [p,(r-1)p')$.
\end{thm}
When $\chi=\delta$, Theorems \ref{main1_thm} and \ref{main3_thm} give the weighted embeddings with the same measure, namely with the left Haar measure (see Theorem \ref{main2_thm} and \ref{main4_thm}), which is
the unique case when such embeddings hold true as in the unweighted case. Moreover, for $q=p$ this gives the following {\bf critical Hardy inequality} on general connected Lie groups:
\begin{equation}\label{lef_rig_Hardy1_crit}
\left\|\frac{f}{\left(\log\left({\rm e}+\frac{1}{|x|}\right)\right)^{\frac{r}{p}}
|x|^{\frac{d}{p}}}\right\|_{L^{p}(\mu_{\chi})}\lesssim \|f\|_{L^{p}_{d/p}( \mu_{\chi})},
\end{equation}
where $1<p<r<\infty$,
which is a critical case $\alpha=d/p$ of the Hardy inequality given in \eqref{lef_rig_Hardy1}.
\begin{rem} First, we will prove the above theorems in the noncompact case. Consequently, in Section \ref{SEC:app_compact}, we show that these Theorems \ref{main1_thm} and \ref{main3_thm} with $\delta=1$ hold on compact Lie groups (which are automatically unimodular) as well.
\end{rem}
We also show that Theorem \ref{main1_thm} gives the following fractional {\bf Caffarelli-Kohn-Nirenberg type inequality}:
\begin{thm}\label{CKN_thm} Let $\G$ be a connected Lie group. Let $e$ be the identity element of $\G$, and let $\chi$ be a positive character of $\G$. Let $|x|:=d_{C}(e,x)$ denote the Carnot-Carath\'{e}odory distance from $e$ to $x$. Let $1<p<\infty$, $0<q,r<\infty$ and $0<\theta\leq 1$ be such that $\theta>(r-q)/r$ and $p\leq q\theta r/(q-(1-\theta)r)$. Let $a$ and $b$ be real numbers and $\alpha>0$ such that $0\leq qr(b(1-\theta)-a)/(q-(1-\theta)r)<d$ and $1/p-(q-(1-\theta)r)/(qr\theta)\leq \alpha/d-(b(1-\theta)-a)/(\theta d)$. Then we have
\begin{equation}\label{CKN}
\||x|^{a}f\|_{L^{r}(\mu_{\chi^{\widetilde{q}/p}\delta^{1-\widetilde{q}/p}})} \lesssim \|f\|^{\theta}_{L^{p}_{\alpha}(\mu_{\chi})}
\||x|^{b}f\|^{1-\theta}_{L^{q}(\mu_{\chi^{\widetilde{q}/p}\delta^{1-\widetilde{q}/p}})},
\end{equation}
where $\widetilde{q}:=\frac{qr\theta}{q-(1-\theta)r}$.
\end{thm}
\begin{rem}
Note that when $\theta=1$ then $\theta>(r-q)/r$ automatically holds, $\widetilde{q}=r$, $p\leq q\theta r/(q-(1-\theta)r)$ gives $p\leq r$, while conditions $0\leq qr(b(1-\theta)-a)/(q-(1-\theta)r)<d$ and $1/p-(q-(1-\theta)r)/(qr\theta)\leq \alpha/d-(b(1-\theta)-a)/(\theta d)$ are equivalent to $0\leq -ar<d$ and $1/p-1/r\leq \alpha/d+a/d$, respectively. Then, in this case, the inequality \eqref{CKN} has the following form
$$\||x|^{a}f\|_{L^{r}(\mu_{\chi^{r/p}\delta^{1-r/p}})} \lesssim \|f\|_{L^{p}_{\alpha}(\mu_{\chi})},
$$
which is \eqref{main1_thm_for1}.
\end{rem}
\begin{rem} We note that if we take $a=b=0$ in \eqref{CKN}, then it gives the Gagliardo-Nirenberg type inequality on general connected Lie
groups: Let $1<p<\infty$, $0<q,r<\infty$, $0<\theta\leq 1$ and $\alpha>0$ be such that $\theta>(r-q)/r$, $p\leq q\theta r/(q-(1-\theta)r)$ and $1/p-(q-(1-\theta)r)/(qr\theta)\leq \alpha/d$. Then we have the following {\bf Gagliardo-Nirenberg inequality}:
\begin{equation}\label{CKN_GN_cor}
\|f\|_{L^{r}(\mu_{\chi^{\widetilde{q}/p}\delta^{1-\widetilde{q}/p}})} \lesssim \|f\|^{\theta}_{L^{p}_{\alpha}(\mu_{\chi})}
\|f\|^{1-\theta}_{L^{q}(\mu_{\chi^{\widetilde{q}/p}\delta^{1-\widetilde{q}/p}})},
\end{equation}
where $\widetilde{q}:=\frac{qr\theta}{q-(1-\theta)r}$.
\end{rem}

Similarly, from \eqref{main3_thm_for1} we can obtain the inequality \eqref{CKN} in the {\bf critical case $a=b(1-\theta)-d(q-(1-\theta)r)/qr$}:
 \begin{thm}\label{CKN_thm_crit} Let $\G$ be a connected Lie group. Let $\chi$ be a positive character of $\G$. Let $b\in\mathbb{R}$, $1<p<r<\infty$, $0<q,r<\infty$ and $0<\theta\leq 1$ be such that $\theta>(r-q)/r$ and $p\leq \widetilde{q}<(r-1)p^{\prime}$ with $p^{\prime}=p/(p-1)$ and $\widetilde{q}:=\frac{qr\theta}{q-(1-\theta)r}$. Then we have
\begin{equation}\label{crit_CKN}
\|\omega_{r}^{b(1-\theta)-d\theta/\widetilde{q}}f\|_{L^{r}( \mu_{\chi^{\widetilde{q}/p}\delta^{1-\widetilde{q}/p}})} \lesssim \|f\|^{\theta}_{L^{p}_{d/p}(\mu_{\chi})}
\|\omega_{r}^{b}f\|^{1-\theta}_{L^{q}( \mu_{\chi^{\widetilde{q}/p}\delta^{1-\widetilde{q}/p}})},
\end{equation}
where $\omega_{r}:=(\log({\rm e}+1/|x|)^{\frac{r}{d}}|x|$.
\end{thm}
We also introduce the following {\bf Hardy-Littlewood-Sobolev inequality}:
\begin{thm}\label{HLS_thm} Let $\G$ be a connected Lie group. Let $e$ be the identity element of $\G$, and let $\chi$ be a positive character of $\G$. Let $|x|:=d_{C}(e,x)$ denote the Carnot-Carath\'{e}odory distance from $e$ to $x$. Let $1<p,q<\infty$, $\alpha\geq 0$ and $0\leq \beta<d/q$. Let $0\leq a_{1}<dp/(p+q)$, $a_{2}>0$ with $0\leq 1/p-q/(p+q)\leq \alpha/d$ and $1/q-p/(p+q)\leq (a_{2}-a_{1})/d$. Then we have
\begin{equation}\label{HLS_thm_ineq}
\left|\int_{\G}\int_{\G}\frac{\overline{f(x)}g(y)G_{a_{2},\chi}^{c}(y^{-1}x)}{|x|^{a_{1}}|y|^{\beta}}d\mu_{\chi^{(p+q)/pq}\delta^{1-(p+q)/pq}}(x)d\rho(y)\right|\lesssim \|f\|_{L^{p}_{\alpha}(\mu_{\chi})}\|g\|_{L^{q}_{\beta}(\mu_{\chi})},
\end{equation}
where $G_{a_{2},\chi}^{c}$ is defined in \eqref{G_bessel_rep}. In particular, $G_{a_{2},\chi}^{c}$ is the convolution kernel of the operator $(\Delta_{\chi}+cI)^{-a_{2}/2}$, i.e. $(\Delta_{\chi}+cI)^{-a_{2}/2}f=f\ast G_{a_{2},\chi}^{c}$.
\end{thm}

Moreover, we show that Theorems \ref{main1_thm} and \ref{main3_thm} imply the following {\bf uncertainty type principles}:
\begin{cor}\label{uncer_cor} Let $\G$ be a connected Lie group. Let $\chi$ be a positive character of $\G$. Let $1/p+1/p'=1$ and $1/q+1/q'=1$.
\begin{itemize} \item If $0\leq \beta<d$, $1<p,q<\infty$ and $\alpha>0$, then we have
\begin{equation}\label{uncer_1}
\|f\|_{L^{p}_{\alpha}(\mu_{\chi})}\||x|^{\frac{\beta}{q}}f\|_{L^{q^{\prime}}(\mu_{\chi^{q/p}\delta^{1-q/p}})}\gtrsim
\|f\|^{2}_{L^{2}(\mu_{\chi^{q/p}\delta^{1-q/p}})}
\end{equation}
for all $q\geq p$ such that $1/p-1/q\leq \alpha/d-\beta/(dq)$, where $|x|:=d_{C}(e,x)$ is the Carnot-Carath\'{e}odory distance from the identity element $e$ to $x$;
\item If $1<p<r<\infty$, then we have
\begin{equation}\label{uncer_2}
\|f\|_{L^{p}_{d/p}( \mu_{\chi})}\left\|\left(\log\left({\rm e}+\frac{1}{|x|}\right)\right)^{\frac{r}{p}}
|x|^{\frac{d}{p}}f\right\|_{L^{q^{\prime}}(B_1, \mu_{\chi^{q/p}\delta^{1-q/p}})}\gtrsim
\|f\|^{2}_{L^{2}(\mu_{\chi^{q/p}\delta^{1-q/p}})}
\end{equation}
for all $q\in [p,(r-1)p')$.
\end{itemize}
\end{cor}

On compact Lie groups, note that all the above results still hold true with $\delta=1$, which we will discuss in Section \ref{SEC:app_compact}.

The main novelty of the paper is the extension to the case of general connected Lie groups and
to the case of sub-Laplacians with drift, of results proved by various people, including the authors
of this paper, in the cases of stratified, and general nilpotent Lie groups, and also
compact groups, in the case of sub-Laplacians.

The organisation of the paper is as follows. In Section \ref{SEC:prelim} we briefly recall some known properties of Sobolev spaces on connected Lie groups. Then, in Section \ref{SEC:Weight_Sob} we prove Theorems \ref{main1_thm}, \ref{main3_thm}, \ref{CKN_thm}, \ref{CKN_thm_crit}, \ref{HLS_thm} and Corollary \ref{uncer_cor}. Finally, in Section \ref{SEC:app_compact} we discuss the obtained results of Section \ref{SEC:Weight_Sob} on compact Lie groups.
\section{Preliminaries}
\label{SEC:prelim}
In this section we very briefly recall some known properties of Sobolev spaces on connected Lie groups.

Let $\G$ be a noncompact connected Lie group with identity $e$. Let us denote the right and left Haar measure by $\rho$ and $\lambda$, respectively. Let $\delta$ be the modular function, i.e. the function on $\G$ such that
\begin{equation}\label{lam_rho}
d \lambda=\delta d \rho.
\end{equation}
Then, recall that $\delta$ is a smooth positive character of $\G$, i.e. a smooth homomorphism of $\G$ into the multiplicative group $\mathbb{R}^{+}$. Let $\chi$ be a continuous positive character of $\G$, which is then automatically smooth. Let $\mu_{\chi}$ be a measure with density $\chi$ with respect to $\rho$,
\begin{equation}\label{mu_rho}
d\mu_{\chi}=\chi d\rho.
\end{equation}
Then, by above \eqref{lam_rho} and \eqref{mu_rho} we see that $\mu_{\delta}=\lambda$ and $\mu_{1}=\rho$.

Let $X=\{X_{1},\ldots, X_{n}\}$ be a family of left-invariant, linearly independent vector
fields which satisfy H\"{o}rmander's condition. We recall that these vector fields induce the Carnot-Carath\'{e}odory distance $d_{C}(\cdot,\cdot)$. Let $B=B(c_{B}, r_{B})$ be a ball with respect to such distance, where $c_{B}$ and $r_{B}$ are its centre and radius, respectively. We write $|x|:=d_{C}(e,x)$. If $V(r)=\rho(B_{r})$ is the volume of the ball $B(e, r)=:B_{r}$ with respect to the right Haar measure $\rho$, then it is well-known (see e.g. \cite{Gui73} or \cite{Var88}) that there exist two constants $d\in \mathbb{N}^{*}$ and $D>0$ such that
\begin{equation}\label{V_less1}
V(r)\approx r^{d}\;\;\;\forall r\in(0,1],
\end{equation}
\begin{equation}\label{V_great1}
V(r)\lesssim {\rm e}^{Dr}\;\;\;\forall r\in(1,\infty).
\end{equation}
We say that $d$ and $D$ are local and global dimensions of the metric measure space $(\G, d_{C}, \rho)$, respectively. Recall that $d$ is uniquely determined by $\G$ and $X$, while the set of $D>0$ such that \eqref{V_great1} holds is independent of $X$ but does not have a minimum in general, see e.g. \cite[page 285]{CRT01}, \cite[Chapter 4]{VCS92} or \cite{BPV21}. We fix a $D>0$ such that \eqref{V_great1} holds. One can observe that the metric measure space $(\G,d_{C},\rho)$ is locally doubling, but not doubling in general. 

We shall denote by $\Delta_{1}$ the smallest self-adjoint extension on $L^{2}(\rho)$ of the \enquote{sum-of-squares} operator
$$\Delta_{1}=-\sum_{j=1}^{n}X_{j}^{2}$$
on $C^{\infty}_{0}(\G)$. We shall denote with $P_{t}(\cdot,\cdot)$ and $p_{t}$ the smooth integral kernel of ${\rm e}^{-t\Delta_{1}}$ and its smooth convolution kernel (i.e. ${\rm e}^{-t\Delta_{1}}f=f\ast p_{t}$) with respect to the measure $\rho$, respectively, where $\ast$ is the convolution between two functions $f$ and $g$ (when it exists), i.e.
$$f\ast g(x) =\int_{\G}f(xy^{-1})g(y)d\rho(y).$$
Recall the following relation
$$P_{t}(x, y)=p_{t}(y^{-1}x)\delta(y) \;\;\forall x, y \in \G.$$

It is also known that the generated semigroup ${\rm e}^{-t\Delta_{\chi}}$ on $L^{2}(\mu_{\chi})$ admits an integral kernel $P_{t}^{\chi}\in \mathcal{D}^{\prime}(\G\times \G)$ with respect to the measure $\mu_{\chi}$
$${\rm e}^{-t\Delta_{\chi}}f(x)=\int_{\G}P_{t}^{\chi}(x,y)f(y)d\mu_{\chi}(y),$$
and admits a convolution kernel $p_{t}^{\chi}\in \mathcal{D}^{\prime}(\G)$
$${\rm e}^{-t\Delta_{\chi}}f(x)=f\ast p_{t}^{\chi}(x)=\int_{\G}f(xy^{-1})p_{t}^{\chi}(y)d\rho(y).$$
For $P_{t}^{\chi}$ and $p_{t}^{\chi}$ we have
$$P_{t}^{\chi}(x,y)=p_{t}^{\chi}(y^{-1}x)\chi^{-1}(y)\delta(y).$$
Denoting $b_{X}:=\frac{1}{2}\left(\sum_{i=1}^{n}c_{i}^{2}\right)^{1/2}$ with $c_{i}=(X_{i}\chi)(e)$, we also have
\begin{equation}\label{pt_chi_pt}
p_{t}^{\chi}(x)={\rm e}^{-tb_{X}^{2}}p_{t}(x)\chi^{-1/2}(x),
\end{equation}
so that $P_{t}^{\chi}$ and $p_{t}^{\chi}$ are smooth on $\G\times \G$ and $\G$, respectively.

According to \cite{BPTV18}, we now recall some useful properties of $L^{p}_{\alpha}(\mu_{\chi})$. For every $1<p<\infty$, $\alpha\geq 0$ and $c>0$, we have
\begin{equation}\label{Sob_norm_Bes1}
\|f\|_{L^{p}_{\alpha}(\mu_{\chi})}\approx\|(\Delta_{\chi}+cI)^{\alpha/2}f\|_{L^{p}(\mu_{\chi})}
\end{equation}
and
$$\|(\Delta_{\chi}+cI)^{\alpha_{2}/2}f\|_{L^{p}(\mu_{\chi})}\leq
\|(\Delta_{\chi}+cI)^{\alpha_{1}/2}f\|_{L^{p}(\mu_{\chi})}$$
 when $\alpha_{1}>\alpha_{2}$, i.e. $L^{p}_{\alpha_{1}}(\mu_{\chi})\hookrightarrow L^{p}_{\alpha_{2}}(\mu_{\chi})$.

Denote by $\mathfrak{I}$ the set $\{1,\ldots,n\}$. Let $\mathfrak{I}^{m}$ be the set of multi-indices $J=(j_{1},\ldots,j_{m})$ such that $j_{i}\in \mathfrak{I}$ for every $m, {i}\in \mathbb{N}$, and let $X_{J}$ be the left differential operator $X_{J}=X_{j_{1}}\ldots X_{j_{m}}$ for $J\in \mathfrak{I}^{m}$.

\begin{prop}\cite[Propositions 3.3 and 3.4]{BPTV18} \label{prop3.3-3.4}
\begin{itemize}
\item Let $k\in \mathbb{N}$ and $1<p<\infty$. Then we have
$$\|f\|_{L^{p}_{k}(\mu_{\chi})}\approx \sum_{J\in \mathfrak{I}^{m}, m\leq k}\|X_{J}f\|_{L^{p}(\mu_{\chi})}.$$
\item For every $\alpha\geq 0$ and $1<p<\infty$ we have
$$f\in L^{p}_{\alpha+1}(\mu_{\chi})\Leftrightarrow f\in L^{p}_{\alpha}(\mu_{\chi})\;\;\text{and}\;\;X_{i}f\in L^{p}_{\alpha}(\mu_{\chi})$$
for every $i\in \mathfrak{I}$. In particular
$$\|f\|_{L^{p}_{\alpha+1}(\mu_{\chi})}\approx \|f\|_{L^{p}_{\alpha}(\mu_{\chi})}+\sum_{i=1}^{n}\|X_{i}f\|_{L^{p}_{\alpha}(\mu_{\chi})}.$$
\end{itemize}
\end{prop}
\begin{prop}\cite[Proposition 5.7 (ii)]{HMM05} \label{prop5.7} Let $\G$ be a noncompact connected Lie group. Let $\|X\|=\left(\sum_{i=1}^{n}c_{i}^{2}\right)^{1/2}$ with $c_{i}=(X_{i}\chi)(e)$, $i\in\mathfrak{I}$. Then for every $r\in \mathbb{R}^{+}$ we have
$$
\sup_{x\in B_{r}}\chi(x)={\rm e}^{\|X\|r}.
$$
\end{prop}
We also recall that for every character $\chi$ and $R>0$, there exists a constant $c=c(\chi,R)$ such that
\begin{equation}\label{double_charac}
c^{-1}\chi(x)\leq \chi(y)\leq c\chi(x)\;\;\forall x,y\in \G \;\;\text{s.t.}\;\;d_{C}(x,y)\leq R,
\end{equation}
which means that the metric measure space $(\G,d_{C},\mu_{\chi})$ is locally doubling.
\begin{lem}\cite[Lemma 2.3]{BPTV18} \label{lem2.3} Let $\G$ be a noncompact connected Lie group. Then we have
\begin{enumerate}[label=(\roman*)]
\item ${\rm e}^{-t\Delta_{\chi}}$ is a diffusion semigroup on $(\G,\mu_{\chi})$;
\item Let $\chi$ be a positive character of $\G$. Then we have  $\int_{B_{r}}\chi d\rho\leq {\rm e}^{(\|X\|+D)r}$ for every $r>1$, where $\|X\|=\left(\sum_{i=1}^{n}c_{i}^{2}\right)^{1/2}$ with $c_{i}=(X_{i}\chi)(e)$, $i\in\mathfrak{I}$.
\item Furthermore, there exist two positive constants $\omega$ and $b$ such that, for every $m\in \mathbb{N}$ and $J\in \mathfrak{I}^{m}$, we have $|X_{J}p_{t}^{\chi}(x)|\lesssim \chi^{-1/2}(x)t^{-(d+m)/2}{\rm e}^{\omega t}{\rm e}^{-b|x|^{2}/t}$, for all $t>0$ and $x\in \G$.
\end{enumerate}
\end{lem}
By virtue of the next proposition, proofs of Theorems \ref{main1_thm} and \ref{main3_thm} can be reduced to proofs of Theorems \ref{main2_thm} and \ref{main4_thm}, respectively:
\begin{prop}\cite[Proposition 3.5]{BPTV18} \label{prop3.5}
Let $p\in(1,\infty)$ and $\alpha\geq 0$. Then we have
$$\|f\|_{L^{p}_{\alpha}(\mu_{\chi})}\approx \|\chi^{1/p}f\|_{L^{p}_{\alpha}(\rho)}.$$
\end{prop}
We note (see also \cite{BPTV18}) that the function
\begin{equation}\label{G_bessel_rep}
G_{\alpha,\chi}^{c}(x)=C(\alpha)\int_{0}^{\infty}t^{\alpha/2-1}{\rm e}^{-ct}p_{t}^{\chi}(x)dt
\end{equation}
is the convolution kernel of the operator $(\Delta_{\chi}+cI)^{-\alpha/2}$, i.e.
\begin{equation}\label{Sob_norm_Bes2}
(\Delta_{\chi}+cI)^{-\alpha/2}f=f\ast G_{\alpha,\chi}^{c},
\end{equation}
where $c>\omega$ and $\omega$ is from Lemma \ref{lem2.3}.
\begin{lem}\cite[Lemma 4.1]{BPTV18}\label{lemma4.1_orig} Let $b$ and $\omega$ be as in Lemma \ref{lem2.3}. Let $c>\omega$ and $c^{\prime}=\frac{1}{2}\sqrt{b(c-\omega)}$. Then we have
\begin{equation}\label{less1_orig}
|G_{\alpha,\chi}^{c}|\leq C
\begin{cases} {}|x|^{\alpha-d} \;\text{\;if}\;\;0<\alpha<d\;\;\text{and}\;\;|x|\leq1,\\
\chi^{-1/2}(x){\rm e}^{-c'|x|}\;\text{\;when}\;\;|x|>1\end{cases}
\end{equation}
for some positive constant $C$.
\end{lem}

We will also use Young's inequalities in the following form:
\begin{lem}\cite[Lemma 4.3]{BPTV18}\label{lemma4.3} Let $1<p\leq q<\infty$ and $r\geq1$ be such that $1/p+1/r=1+1/q$. Then we have
\begin{equation}\label{Young_for}
\|f \ast g\|_{L^{q}(\lambda)}\leq \|f\|_{L^{p}(\lambda)}(\|\check{g}\|^{r/p^{\prime}}_{L^{r}(\lambda)}\|g\|^{r/q}_{L^{r}(\lambda)}),
\end{equation}
where $\check{g}(x)=g(x^{-1})$.
\end{lem}
\begin{rem} For a simpler version of Young's inequality on general locally compact groups we refer to \cite[cf. Lemma 2.1]{KR78}.
\end{rem}

The following integral Hardy inequalities on general metric measure spaces, which are the special cases of \cite[Theorems 2.1 and 3.1 (a)]{Myn23} (see also \cite{Sin22}), play important role in the proof of the main results:
\begin{thm}\label{high_Hardy_thm} Let $\mathbb{X}$ be a metric measure space with a $\sigma$-finite measure. Let $0$ be a fixed element of $\mathbb{X}$ and $|x|=d(0, x)$. Let $1<p\leq q<\infty$. Let $\{\phi_{i}\}_{i=1}^{2}$ and $\{\psi_{i}\}_{i=1}^{2}$ be positive functions on $\X$. Then the inequalities
\begin{equation}\label{integ_Hardy1}
\left(\int_{\X}\left(\int_{B(0,{|x|/2})}f(z)dz\right)^{q}\phi_{1}(x)dx\right)^{\frac{1}{q}}\leq A_{1}
\left(\int_{\X}(f(x))^{p}\psi_{1}(x)dx\right)^{\frac{1}{p}}
\end{equation}
and
\begin{equation}\label{integ_Hardy2}
\left(\int_{\X}\left(\int_{\X\backslash B(0,{2|x|})}f(z)dz\right)^{q}\phi_{2}(x)dx\right)^{\frac{1}{q}}\leq A_{2}
\left(\int_{\X}(f(x))^{p}\psi_{2}(x)dx\right)^{\frac{1}{p}}
\end{equation}
hold for all $f\geq 0$ a.e. on $\X$ if we have
\begin{equation}\label{integ_Hardy1_cond}
B_{1}:=\sup_{R>0}\left(\int_{\{|x|\geq R\}}\phi_{1}(x)dx\right)^{\frac{1}{q}}
\left(\int_{\{|x|< R\}}(\psi_{1}(x))^{1-p'}dx\right)^{\frac{1}{p'}}<\infty
\end{equation}
and
\begin{equation}\label{integ_Hardy2_cond}
B_{2}:=\sup_{R>0}\left(\int_{\{|x|\leq R\}}\phi_{2}(x)dx\right)^{\frac{1}{q}}
\left(\int_{\{|x|\geq R\}}(\psi_{2}(x))^{1-p'}dx\right)^{\frac{1}{p'}}<\infty,
\end{equation}
respectively. Moreover, if $\{A_{i}\}_{i=1}^{2}$ are the smallest constants for which \eqref{integ_Hardy1} and \eqref{integ_Hardy2} hold, then
\begin{equation}\label{integ_Hardy5}
c_{1}^{i}B_{i}\leq A_{i}\leq c_{2}^{i}B_{i},\;\; i=1,2.
\end{equation}
\end{thm}
\begin{rem} In the setting of metric measure spaces, first Theorem \ref{high_Hardy_thm} was proved in \cite{RV18} on metric measure spaces possessing polar decompositions. We can also refer to \cite{AR24} for the analysis of polar decompositions in metric measure spaces. To avoid such technicalities, we will be using this result as it follows from \cite{Myn23}.
\end{rem}

\section{Proof of Main results}
\label{SEC:Weight_Sob}
\subsection{Proof of Theorem \ref{main1_thm}}
\label{SEC:Weight_Sob1}
In this section we prove Theorems \ref{main1_thm}, \ref{main3_thm}, \ref{CKN_thm}, \ref{CKN_thm_crit} and Corollary \ref{uncer_cor} when $\G$ is noncompact, and in the case when $\G$ is compact we refer to Section \ref{SEC:app_compact} for the differences in the argument in this setting.

Before starting the proof, we need to prove the following lemma:
\begin{lem}\label{Bes_est_lem}
Let $a,s\in \mathbb{R}$ and $r>0$. If $c'>0$ is sufficiently large, then we have
\begin{equation}\label{Bes_est_lem_for1}\int_{B_{1}^{c}}|\delta^{a}\chi^{s}{\rm e}^{-c'|x|}|^{r}d\rho<\infty.
\end{equation}
\end{lem}
Actually, the proof of this lemma follows from the proof of \cite[Corollary 4.2]{BPTV18}, but to be more precise, let us give it.
\begin{proof}[Proof of Lemma \ref{Bes_est_lem}] Taking into account Lemma \ref{lem2.3} (ii), a direct calculation gives that
\begin{equation*}
\begin{split}
\int_{B_{1}^{c}}|\delta^{a}\chi^{s}{\rm e}^{-c'|x|}|^{r}d\rho&=\sum_{k=0}^{\infty}{\rm e}^{-rc^{\prime}2^{k}}\int_{\{2^{k}\leq|x|<
2^{k+1}\}}
(\delta(x))^{ra}(\chi(x))^{rs}d\rho(x)\\&
\lesssim\sum_{k=0}^{\infty}{\rm e}^{-rc^{\prime}2^{k}}{\rm e}^{C\cdot2^{k}}<\infty,
\end{split}
\end{equation*}
since $c^{\prime}$ is large enough.
\end{proof}
Once we prove the special case $\chi=\delta$ of Theorem \ref{main1_thm}, then we can immediately obtain Theorem \ref{main1_thm} by Proposition \ref{prop3.5}. Therefore, let us prove the following theorem:
\begin{thm}\label{main2_thm}
Let $\alpha>0$, $0\leq \beta<d$ and $1<p,q<\infty$. Then we have
\begin{equation}\label{main2_thm_for1}
\left\|\frac{f}{|x|^{\frac{\beta}{q}}}\right\|_{L^{q}(\lambda)}\lesssim \|f\|_{L^{p}_{\alpha}(\lambda)}
\end{equation}
for all $q\geq p$ such that $1/p-1/q\leq \alpha/d-\beta/(dq)$.
\end{thm}
\begin{rem}\label{rem_only_if_left} When $\beta=0$, in \cite[Section 4]{BPTV18}, it was shown that an embedding of the form
$$L^{p}_{\alpha}(\mu_{\chi})\hookrightarrow L^{q}(\mu_{\chi}),\;1<p<\infty, \;\alpha\geq0, \;q\in(1,\infty]\backslash\{p\}$$
for some positive character $\chi$ may hold only if $\mu_{\chi}=\lambda$ is the left Haar measure of $\G$. In the exactly same
way, one can show that the same statement is true for the weighted Sobolev embedding case. However, this is different for $q=p$, see \eqref{lef_rig_Hardy1}.
\end{rem}
\begin{proof}[Proof of Theorem \ref{main2_thm}] Notice that we may reduce to prove Theorem \ref{main2_thm} when $0<\alpha<d$, since when $\alpha \geqslant d$ we may find $0<\alpha^{\prime}<d$ such that
$$
\frac{\alpha^{\prime}}{d}-\frac{\beta}{d q} \geqslant 1-\frac{1}{q}>\frac{1}{p}-\frac{1}{q},
$$
so that the condition $1 / p-1 / q \leqslant \alpha^{\prime} / d-\beta /(d q)$ holds. Then, we apply Theorem \ref{main2_thm} to $\alpha^{\prime}\in (0,d)$ and use the Sobolev embedding $L_\alpha^p(\lambda) \subseteq L_{\alpha^{\prime}}^p(\lambda)$ (see \eqref{Sob_emb_unimod1} when $q=p$ and $\mu_{\chi}=\lambda$). Therefore, it is enough to prove Theorem \ref{main2_thm} when $0<\alpha<d$.

By \eqref{Sob_norm_Bes1} and \eqref{Sob_norm_Bes2}, we note that to obtain \eqref{main2_thm_for1} it is enough to prove the following
$$\int_{\G}|(f\ast G^{c}_{\alpha,\chi})(x)|^{q}\frac{d\lambda(x)}{|x|^{\beta}}\lesssim \|f\|^{q}_{L^{p}(\lambda)}.$$

For this, let us split the left-hand side of above inequality into three parts as follows
\begin{equation}\label{M123}
\int_{\G}|(f\ast G^{c}_{\alpha,\chi})(x)|^{q}\frac{d\lambda(x)}{|x|^{\beta}}\leq 3^{q}(M_{1}+M_{2}+M_{3}),
\end{equation}
where
$$M_{1}:=\int_{\G}\left(\int_{\{2|y|<|x|\}}| G^{c}_{\alpha,\chi}(y^{-1}x)f(y)|d\lambda(y)\right)^{q}
\frac{d\lambda(x)}{|x|^{\beta}},$$
$$M_{2}:=\int_{\G}\left(\int_{\{|x|\leq2|y|<4|x|\}}| G^{c}_{\alpha,\chi}(y^{-1}x)f(y)|d\lambda(y)\right)^{q}
\frac{d\lambda(x)}{|x|^{\beta}}$$
and
$$M_{3}:=\int_{\G}\left(\int_{\{|y|\geq 2|x|\}}|G^{c}_{\alpha,\chi}(y^{-1}x)f(y)|d\lambda(y)\right)^{q}
\frac{d\lambda(x)}{|x|^{\beta}}.$$
Let us start by estimating the first term $M_{1}$. By using the reverse triangle inequality and $2|y|<|x|$ we have
\begin{equation}
\label{quasi_Euc_norm1}
\begin{split}
&|y^{-1}x|\geq |x|-|y|>|x|-\frac{|x|}{2}=\frac{|x|}{2},\\&
|y^{-1}x|\leq |x|+|y|<\frac{3|x|}{2}.
\end{split}
\end{equation}
Taking into account this, for $M_{1}$ we write 
\begin{equation*}
M_{1}\leq
\int_{\G}\left(\int_{\{2|y|<|x|\}}|f(y)|d\lambda(y)\right)^{q}\left(\sup_{\{|x|<2|z|<3|x|\}}|G^{c}_{\alpha,\chi}(z)|\right)^{q}
\frac{d\lambda(x)}{|x|^{\beta}}.
\end{equation*}
In order to apply the integral Hardy inequality \eqref{integ_Hardy1}, let us check the following condition \eqref{integ_Hardy1_cond}:
\begin{equation}\label{check1}
\left(\int_{\{2r_{0}\leq |x|\}}\left(\sup_{\{|x|<2|z|<3|x|\}}|G_{\alpha,\chi}^{c}(z)|\right)^{q}
\frac{d\lambda(x)}{|x|^{\beta}}\right)^{\frac{1}{q}}
\left(\int_{\{|x|< 2r_{0}\}}d\lambda(x)\right)^{\frac{1}{p^{\prime}}}<\infty
\end{equation}
for all $r_{0}>0$. Indeed, once \eqref{check1} has been established, the integral Hardy inequality \eqref{integ_Hardy1} implies 
\begin{equation}\label{M1}
M_{1}^{\frac{1}{q}}\leq C\|f\|_{L^{p}(\lambda)},
\end{equation}
where $C$ does not depend on $f$.

Now, let us check \eqref{check1}. By Lemma \ref{lemma4.1_orig} we have
\begin{equation}\label{A_alp}
\sup_{\{|x|<2|z|<3|x|\}}|G_{\alpha,\chi}^{c}(z)|\\ \leq C_{1}
\begin{cases} |x|^{\alpha-d}\;\text{\;if}\;\;0<\alpha<d\;\text{\;and}\;\;|z|\leq1,\\
{\rm e}^{C_{2}|x|}{\rm e}^{-c'|x|/2}\;\text{\;if}\;\;|z|>1\end{cases}
\end{equation}
where we have used $\underset{\{|x|<2|z|<3|x|\}}{\rm sup}\chi^{-1/2}(z)\leq {\rm e}^{C_{2}|x|}$ by Proposition \ref{prop5.7}. 

For this, we consider the following cases: $r_{0}>1$ and $0<r_{0}\leq 1$.

In the case $r_{0}>1$, we have $2<2r_{0}\leq |x|<2|z|$. Then using \eqref{A_alp} one has
\begin{multline}\label{check2}
\int_{\{2r_{0}\leq |x|\}}\left(\sup_{\{|x|<2|z|<3|x|\}}|G_{\alpha,\chi}^{c}(z)|\right)^{q}
\frac{d\lambda(x)}{|x|^{\beta}}
\leq C_{1}\frac{{\rm e}^{-qc^{\prime}\frac{r_{0}}{4}}}{r_{0}^{\beta}}\int_{\{2r_{0}\leq |x|\}}{\rm e}^{-c^{\prime}q\frac{|x|}{4}}{\rm e}^{C_{2}q|x|}
d\lambda(x)\\
= C_{1}\frac{{\rm e}^{-qc^{\prime}\frac{r_{0}}{4}}}{r_{0}^{\beta}}\sum_{k=0}^{\infty}\int_{\{2^{k+1}r_{0}\leq |x|\leq
2^{k+2}r_{0}\}}{\rm e}^{-c^{\prime}q\frac{|x|}{4}}{\rm e}^{C_{2}q|x|}
d\lambda(x)\\
\leq C_{1}\frac{{\rm e}^{-qc^{\prime}\frac{r_{0}}{4}}}{r_{0}^{\beta}}\sum_{k=0}^{\infty}{\rm e}^{-c^{\prime}q2^{k-1}r_{0}}{\rm e}^{C_{2}2^{k+2}r_{0}q}
\int_{\{2^{k+1}r_{0}\leq |x|\leq
2^{k+2}r_{0}\}}
d\lambda(x)
\lesssim r_{0}^{-\beta}{\rm e}^{-qc^{\prime}\frac{r_{0}}{4}},
\end{multline}
where the sum is finite since $c$ (hence $c^{\prime}$) is large enough.

By Part (ii) of Lemma \ref{lem2.3}, for $r_{0}>1$ we also have
\begin{equation}\label{check3}
\int_{\{|x|< 2r_{0}\}}d\lambda(x)=\int_{\{|x|< 2r_{0}\}}\delta d\rho(x)\leq {\rm e}^{2Cr_{0}}.
\end{equation}

Then, plugging \eqref{check2} and \eqref{check3} into \eqref{check1} we obtain
\begin{multline}\label{check4}
\left(\int_{\{2r_{0}\leq |x|\}}\left(\sup_{\{|x|<2|z|<3|x|\}}|G_{\alpha,\chi}^{c}(z)|\right)^{q}
\frac{d\lambda(x)}{|x|^{\beta}}\right)^{\frac{1}{q}}
\left(\int_{\{|x|< 2r_{0}\}}d\lambda(x)\right)^{\frac{1}{p^{\prime}}}\\\lesssim r_{0}^{-\frac{\beta}{q}}{\rm e}^{-c^{\prime}\frac{r_{0}}{4}}
{\rm e}^{2C\frac{r_{0}}{p^{\prime}}}
<\infty,
\end{multline}
since $c$ (hence $c^{\prime}$) is large enough. 

Now, we check the condition \eqref{check1} for $0<r_{0}\leq 1$. As noticed above, when $|z|>1$ due to the exponential decay $G_{\alpha,\chi}^{c}(z)$ from \eqref{A_alp} we can easily obtain \eqref{check1}. So, let us discuss the case $|z|\leq 1$. In this case, taking into account $|x|<2|z|\leq 2$ we write
\begin{multline}\label{check5}
\int_{\{2r_{0}\leq |x|\}}\left(\sup_{\{|x|<2|z|<3|x|\}}|G_{\alpha,\chi}^{c}(z)|\right)^{q}
\frac{d\lambda(x)}{|x|^{\beta}}
\\=
\int_{\{2r_{0}\leq |x|\leq 1\}}\left(\sup_{\{|x|<2|z|<3|x|\}}|G_{\alpha,\chi}^{c}(z)|\right)^{q}\frac{d\lambda(x)}{|x|^{\beta}}\\
+\int_{\{1<|x|<2\}}\left(\sup_{\{|x|<2|z|<3|x|\}}|G_{\alpha,\chi}^{c}(z)|\right)^{q}\frac{d\lambda(x)}{|x|^{\beta}}.
\end{multline}
Using estimate \eqref{A_alp}, one can observe that the last integral is finite. To estimate the first integral on the right-hand side of \eqref{check5}, we split it into two cases: $(\alpha-d)q-\beta+d\neq 0$ and $(\alpha-d)q-\beta+d= 0$. In the first case, taking into account \eqref{A_alp}, we calculate
\begin{equation}\label{check6}
\begin{split}\int_{\{2r_{0}\leq |x|\leq 1\}}\left( \sup_{\{|x|<2|z|<3|x|\}}|G_{\alpha,\chi}^{c}(z)|\right)^{q}\frac{d\lambda(x)}{|x|^{\beta}}
&\lesssim\int_{\{2r_{0}\leq |x|\leq 1\}}|x|^{(\alpha-d)q-\beta} d\lambda(x)
\\&\lesssim \int_{2r_{0}}^{1}u^{(\alpha-d)q-\beta}u^{d-1}du \\& \lesssim 1+r_{0}^{(\alpha-d)q-\beta+d}.
\end{split}
\end{equation}
For the inequality of passing from the integral with respect to $d\lambda$ to the one with respect to $u$, we first observe that the left Haar measure is absolutely continuous with respect to the Riemannian measure. Indeed, if one considers a full form on the Lie algebra of $\G$, it can be moved around by the group action to yield the left Haar measure on $\G$. By the uniqueness of the left Haar measure, one gets the absolute continuity as above, with the left Haar measure being a smooth multiple of the volume measure. Consequently, $d\lambda$ is absolutely continuous with respect to the radial measure, see e.g. \cite[Corollary 2, p. 81]{AR24}, but the question of the Jacobian remains. However, for an estimate (as opposed to the exact equality), we can give a short direct argument.

Let $B_r$ denote the ball, centred at a fixed point, of radius $r$ with respect to the Carnot-Carath\'eodory distance, that is, $x\in B_r$ if $|x|<r$. Let us introduce the function $s=s(r)$ given by $s(r):=\lambda(B_r)^{1/d}$, where $d\lambda$ is the left Haar measure on $\G$, and $d$ is the local dimension of $\G$. Let us identify the balls with radii given by $s(r)$ and $r$, by writing $\tilde{B}_s=B_r.$ Then we have
$\lambda(\tilde{B}_s)=\lambda(B_r)=s^d$. Since $\lambda(B_r)\leq c r^d,$ we have that 
$$ s^d=\lambda(\tilde{B}_s)=\lambda(B_r)\leq c r^d,$$
that is, $s\leq c r$ for some $c>0$. Consequently, for any $\gamma>0$, we have $r^{-\gamma}\leq c s^{-\gamma}$ for some $c>0$, and $B_{r_0}\subset {B}_{s_0/c}$, for $s_0=s(r_0).$ Consequently, we can estimate, with $r=|x|,$ and using that now we have the equality $\lambda(\tilde{B}_s)=s^d$,
\begin{equation}\label{ccest}
\int_{B_{r_0}} r^{-\gamma} d\lambda(x)\leq C\int_{B_{s_0/c}} s^{-\gamma} d\lambda(x)\leq C\int_{0}^{s_0/c} s^{-\gamma} s^{d-1} ds<\infty,
\end{equation}
provided that $\gamma<d.$ Applying and adapting arguments of this type here and in the sequel, justifies local estimates like the one in
\eqref{check6}.

Taking into account \eqref{check5} and \eqref{check6} in \eqref{check1}, we have for any $0<r_{0}\leq 1$ that
\begin{multline}\label{check8}
\left(\int_{\{2r_{0}\leq |x|\}}\left(\sup_{\{|x|<2|z|<3|x|\}}|G_{\alpha,\chi}^{c}(z)|\right)^{q}
\frac{d\lambda(x)}{|x|^{\beta}}\right)^{\frac{1}{q}}
\left(\int_{\{|x|< 2r_{0}\}}d\lambda(x)\right)^{\frac{1}{p^{\prime}}}\\\leq Cr_{0}^{\frac{d}{p^{\prime}}}
(1+r_{0}^{\frac{(\alpha-d)q-\beta+d}{q}})<\infty
\end{multline}
since $1/p-1/q\leq \alpha/d-\beta/(dq)$.

Now, in the case $(\alpha-d)q-\beta+d=0$, from \eqref{check6} and noting the fact that $r_{0}^{\frac{d}{p^{\prime}}}\left|\log r_0\right|^{\frac{1}{q}}\rightarrow0$ as $r_{0}\rightarrow0$ we have
\begin{multline}\label{check8_11}
\left(\int_{\{2r_{0}\leq |x|\}}\left(\sup_{\{|x|<2|z|<3|x|\}}|G_{\alpha,\chi}^{c}(z)|\right)^{q}
\frac{d\lambda(x)}{|x|^{\beta}}\right)^{\frac{1}{q}}
\left(\int_{\{|x|< 2r_{0}\}}d\lambda(x)\right)^{\frac{1}{p^{\prime}}}\\\leq Cr_{0}^{\frac{d}{p^{\prime}}}
\left|\log r_0\right|^{\frac{1}{q}}<\infty
\end{multline}
for all $0<r_{0}\leq 1$.

Now let us estimate $M_{3}$. Similarly to \eqref{quasi_Euc_norm1}, it is easy to see that the condition $2|x|<|y|$ implies $|y|<2|y^{-1}x|<3|y|$. Then, taking into account this and \eqref{A_alp} we obtain for $M_{3}$ that
$$M_{3}\leq C\int_{\G}\left(\int_{\{|y|\geq 2|x|\}}|f(y)|
\sup_{\{|y|\leq 2|z|\leq 3|y|\}}|G_{\alpha,\chi}^{c}(z)|d\lambda(y)\right)^{q}\frac{d\lambda(x)}
{|x|^{\beta}}.$$
Here, we now apply the conjugate integral Hardy inequality \eqref{integ_Hardy2} for $M_{3}$, for which we need to check the following condition \eqref{integ_Hardy2_cond}:
\begin{equation}\label{check9}
\left(\int_{\{|x|\leq 2r_{0}\}}\frac{d\lambda(x)}{|x|^{\beta}}\right)^{\frac{1}{q}}
\left(\int_{\{2r_{0}\leq |y|\}}\left(
\sup_{\{|y|\leq 2|z|\leq 3|y|\}}|G_{\alpha,\chi}^{c}(z)|\right)^{p^{\prime}}d\lambda(y)\right)^{\frac{1}{p^{\prime}}}<\infty
\end{equation}
for all $r_{0}>0$. Indeed, once \eqref{check9} has been established, the conjugate integral Hardy inequality \eqref{integ_Hardy2} yields 
\begin{equation}\label{M3}
M_{3}^{\frac{1}{q}}\leq C\|f\|_{L^{p}(\lambda)},
\end{equation}
where $C$ does not depend on $f$.

For this, we again consider two cases: $r_{0}>1$ and $0<r_{0}\leq 1$. When $r_{0}>1$, hence $|z|>1$, then as in \eqref{check2} we have
\begin{equation}\label{check10}
\int_{\{2r_{0}\leq |y|\}}\left(\sup_{\{|y|\leq 2|z|\leq 3|y|\}}|G_{\alpha,\chi}^{c}(z)|\right)^{p^{\prime}}
d\lambda(y) \leq C{\rm e}^{-p^{\prime}c^{\prime}\frac{r_{0}}{4}}.
\end{equation}
Applying Part (ii) of Lemma \ref{lem2.3}, one gets for $r_{0}>1$ that
\begin{equation}\label{check11}
\begin{split}
\int_{\{|x|\leq 2r_{0}\}}\frac{d\lambda(x)}{|x|^{\beta}}& \leq \int_{\{|x|\leq 1\}}\frac{d\lambda(x)}{|x|^{\beta}}+
\int_{\{1< |x|\leq 2r_{0}\}}d\lambda(x) \\& \leq C\int_{0}^{1}u^{d-1-\beta}du+
\int_{\{|x|\leq 2r_{0}\}}\delta d\rho(x)\leq C_{3}+{\rm e}^{C_{4}r_{0}}
\end{split}
\end{equation}
for some positive constants $C_{3}$ and $C_{4}$. Then, putting \eqref{check10} and \eqref{check11} in \eqref{check9}, we obtain 
\begin{multline}\label{check12}
\left(\int_{\{|x|\leq 2r_{0}\}}\frac{d\lambda(x)}{|x|^{\beta}}\right)^{\frac{1}{q}}
\left(\int_{\{2r_{0}\leq |y|\}}\left(\sup_{\{|y|\leq 2|z|\leq 3|y|\}}|G_{\alpha,\chi}^{c}(z)|\right)^{p^{\prime}}d\lambda(y)\right)^{\frac{1}{p^{\prime}}}\\\leq C
(C_{3}+{\rm e}^{C_{4}r_{0}})^{\frac{1}{q}}{\rm e}^{-c^{\prime}\frac{r_{0}}{4}}<\infty
\end{multline}
since $c$ (hence $c^{\prime}$) is large enough. 

Now, let us check the condition \eqref{check9} for $0<r_{0}\leq 1$. When $|z|>1$ we readily obtain \eqref{check9} because of the estimate \eqref{A_alp}. In the case $|z|\leq 1$, as in \eqref{check5} and \eqref{check6}, we obtain for $(\alpha-d)p^{\prime}+d\neq 0$ that
\begin{equation}\label{check13}
\int_{\{2r_{0}\leq |y|\}}\left( \sup_{\{|y|\leq 2|z|\leq 3|y|\}}|G_{\alpha,\chi}^{c}(z)|\right)^{p^{\prime}}d\lambda(x)\leq C (1+r_{0}^{(\alpha-d)p^{\prime}+d}).
\end{equation}
Using this in \eqref{check9} implies for $0<r_{0}\leq 1$ that
\begin{multline}\label{check16}
\left(\int_{\{|x|\leq 2r_{0}\}}\frac{d\lambda(x)}{|x|^{\beta}}\right)^{\frac{1}{q}}
\left(\int_{\{2r_{0}\leq |y|\}}\left(\sup_{\{|y|\leq 2|z|\leq 3|y|\}}|G_{\alpha,\chi}^{c}(z)|\right)^{p^{\prime}}d\lambda(y)\right)^{\frac{1}{p^{\prime}}}\\ \leq C
(1+r_{0}^{(\alpha-d)p^{\prime}+d})^{\frac{1}{p^{\prime}}}r_{0}^{\frac{d-\beta}{q}}<\infty
\end{multline}
since $d>\beta$ and $1/p-1/q\leq \alpha/d-\beta/(dq)$.

Note that \eqref{check16} is still finite for $(\alpha-d)p^{\prime}+d=0$ and $0<r_{0}\leq 1$, since as in \eqref{check6} and \eqref{check8_11} we have
\begin{multline}\label{check16_11}
\left(\int_{\{|x|\leq 2r_{0}\}}\frac{d\lambda(x)}{|x|^{\beta}}\right)^{\frac{1}{q}}
\left(\int_{\{2r_{0}\leq |y|\}}\left(\sup_{\{|y|\leq 2|z|\leq 3|y|\}}|G_{\alpha,\chi}^{c}(z)|\right)^{p^{\prime}}d\lambda(y)\right)^{\frac{1}{p^{\prime}}}\\ \leq C
|\log r_{0}|^{\frac{1}{p^{\prime}}}r_{0}^{\frac{d-\beta}{q}}<\infty
\end{multline}
since $|\log r_{0}|^{\frac{1}{p^{\prime}}}r_{0}^{\frac{d-\beta}{q}}\rightarrow0$ as $r_{0}\rightarrow0$ when $d>\beta$.

Now, it remains to estimate $M_{2}$. We rewrite $M_{2}$ as
$$M_{2}=\sum_{k\in\mathbb{Z}}\int_{\{2^{k}\leqslant |x|<2^{k+1}\}}\left(\int_{\{|x|\leqslant 2|y|\leqslant 4|x|\}}|
G_{\alpha,\chi}^{c}(y^{-1}x)f(y)|d\lambda(y)\right)^{q}\frac{d\lambda(x)}{|x|^{\beta}}.$$
We obtain that $2^{k-1}\leqslant |y|<2^{k+2}$ from $|x|\leqslant 2|y|\leqslant 4|x|$ and $2^{k}\leqslant |x|<2^{k+1}$. Let us show that $G_{\alpha,\chi}^{c} \in L^{r}(\lambda)$ for $r\in [1,\infty]$ such that $1+\frac{1}{q}=\frac{1}{r}+\frac{1}{p}$, which is useful in the rest of proof. Indeed, by Lemmata \ref{lemma4.1_orig} and \ref{Bes_est_lem} we see that
\begin{multline}\label{Bes_est_young1}
\int_{\G}| G_{\alpha,\chi}^{c}(x)|^{r}d\lambda(x)=\int_{\{|x|<1\}}| G_{\alpha,\chi}^{c}(x)|^{r}d\lambda(x)
+\int_{\{|x|\geq1\}}| G_{\alpha,\chi}^{c}(x)|^{r}d\lambda(x)\\ \leq 
\mathfrak{C}_{1}\int_{0}^{1}u^{(\alpha-d)r}u^{d-1}du+\mathfrak{C}_{2}\int_{\{|x|\geq1\}}(\chi(x))^{-r/2}{\rm e}^{-c^{\prime}r|x|}\delta d\rho(x)<\infty
\end{multline}
for some positive $\mathfrak{C}_{1}$ and $\mathfrak{C}_{2}$, since $\alpha>d(1/p-1/q)$ and $c$ (hence $c^{\prime}$) is large enough. Similarly, one can show that
\begin{equation}\label{Bes_est_young2}
\|\check{G}_{\alpha,\chi}^{c}\|_{L^{r}(\lambda)}<\infty,
\end{equation}
where $\check{G}_{\alpha,\chi}^{c}(x)=G_{\alpha,\chi}^{c}(x^{-1})$.

Then, taking into account \eqref{Bes_est_young1} and \eqref{Bes_est_young2}, and applying Young's inequality \eqref{Young_for}
for $1+\frac{1}{q}=\frac{1}{r}+\frac{1}{p}$ with $r\in [1,\infty]$ we calculate
\begin{equation}\label{M22}
\begin{split}
M_{2}&\leq C \sum_{k\in\mathbb{Z}}\int_{\{2^{k}\leqslant |x|<2^{k+1}\}}\left(\int_{\{|x|\leqslant 2|y|\leq 4|x|\}}|
G_{\alpha,\chi}^{c}(y^{-1}x)f(y)|d\lambda(y)\right)^{q}d\lambda(x)\\&
\leq C\sum_{k\in\mathbb{Z}} \|[f\cdot\chi_{\{2^{k-1}\leqslant |\cdot|<2^{k+2}\}}]\ast G_{\alpha,\chi}^{c}\|^{q}_{L^{q}(\lambda)}\\& \leq C\|
\check{G}_{\alpha,\chi}^{c}\| ^{qr/p^{\prime}}_{L^{r}(\lambda)}
\| G_{\alpha,\chi}^{c}\| ^{r}_{L^{r}(\lambda)}
\sum_{k\in\mathbb{Z}}\|f\cdot\chi_{\{2^{k-1}\leqslant |\cdot|<2^{k+2}\}}\|^{q}_{L^{p}(\lambda)}\\&
=C \sum_{k\in\mathbb{Z}}\left(\int_{\{2^{k}\leqslant |x|<2^{k+1}\}}|f(x)|^{p}d\lambda(x)\right)^{\frac{q}{p}}\\& 
=C\|f \|^{q}_{L^{p}(\lambda)}.
\end{split}
\end{equation}

Thus, \eqref{M1}, \eqref{M3}, \eqref{M22} and \eqref{M123} complete the proof of Theorem
\ref{main2_thm}.
\end{proof}
\subsection{Proof of Theorem \ref{main3_thm}}
\label{SEC:Weight_Sob2}
We now prove the critical case $\beta=d$ of Theorem \ref{main1_thm} on $B_{1}$.

As in Section \ref{SEC:Weight_Sob1}, we first show the special case $\chi=\delta$ of Theorem \ref{main3_thm}, that is,
\begin{thm}\label{main4_thm}
Let $1<p<r<\infty$ and $1/p+1/p'=1$.
Then we have
\begin{equation}\label{main4_thm_for1}
\left\|\frac{f}{\left(\log\left({\rm e}+\frac{1}{|x|}\right)\right)^{\frac{r}{q}}
|x|^{\frac{d}{q}}}\right\|_{L^{q}( \lambda)}\lesssim \|f\|_{L^{p}_{d/p}(\lambda)}
\end{equation}
for every $q\in [p,(r-1)p')$.
\end{thm}
Once we prove Theorem \ref{main4_thm}, then by Proposition \ref{prop3.5} we obtain immediately Theorem \ref{main3_thm}.
Therefore, we only prove Theorem \ref{main4_thm}.
\begin{proof}[Proof of Theorem \ref{main4_thm}] As in the proof of Theorem \ref{main2_thm}, we split the integral into three
parts
\begin{equation}\label{N123}
\int_{\G}|(f\ast G^{c}_{d/p,\chi})(x)|^{q}\frac{d\lambda(x)}{\left|\log\left({\rm e}+\frac{1}{|x|}\right)\right|^{r}|x|^{d}}\leq
3^{q}(N_{1}+N_{2}+N_{3}),
\end{equation}
where
$$N_{1}:=\int_{\G}\left(\int_{\{2|y|<|x|\}}| G^{c}_{d/p,\chi}(y^{-1}x)f(y)|d\lambda(y)\right)^{q}
\frac{d\lambda(x)}{\left|\log\left({\rm e}+\frac{1}{|x|}\right)\right|^{r}|x|^{d}},$$
$$N_{2}:=\int_{\G}\left(\int_{\{|x|\leq2|y|<4|x|\}}| G^{c}_{d/p,\chi}(y^{-1}x)f(y)|d\lambda(y)\right)^{q}\frac{d\lambda(x)}
{\left|\log\left({\rm e}+\frac{1}{|x|}\right)\right|^{r}|x|^{d}}$$
and
$$N_{3}:=\int_{\G}\left(\int_{\{|y|\geq 2|x|\}}| G^{c}_{d/p,\chi}(y^{-1}x)f(y)|d\lambda(y)\right)^{q}\frac{d\lambda(x)}
{\left|\log\left({\rm e}+\frac{1}{|x|}\right)\right|^{r}|x|^{d}}.$$
First, we estimate $N_{1}$. As in the case of $M_{1}$, taking into account \eqref{quasi_Euc_norm1} and \eqref{A_alp} we have
\begin{multline}\label{N1}
N_{1}\leq \\ 
\int_{\G}\left(\int_{\{2|y|<|x|\}}|f(y)|d\lambda(y)\right)^{q}\left(\sup_{\{|x|<2|z|<3|x|\}}|G^{c}_{d/p,\chi}(z)|\right)^{q}
\frac{d\lambda(x)}{\left|\log\left({\rm e}+\frac{1}{|x|}\right)\right|^{r}|x|^{d}}.
\end{multline}
Here, we will apply the integral Hardy inequality \eqref{integ_Hardy1}, for which we need to check the following condition \eqref{integ_Hardy1_cond}:  
\begin{equation}\label{crit_check1}
\left(\int_{\{2r_{0}\leq |x|\}}\left(\sup_{\{|x|<2|z|<3|x|\}}|G^{c}_{d/p,\chi}(z)|\right)^{q}
\frac{d\lambda(x)}{\left|\log\left({\rm e}+\frac{1}{|x|}\right)\right|^{r}|x|^{d}}\right)^{\frac{1}{q}}
\left(\int_{\{|x|< 2r_{0}\}}d\lambda(x)\right)^{\frac{1}{p^{\prime}}}<\infty
\end{equation}
holds for all $r_{0}>0$. Indeed, once \eqref{crit_check1} has been established, the integral Hardy inequality \eqref{integ_Hardy1} gives
\begin{equation}\label{crit_check5}
N_{1}^{\frac{1}{q}}\leq C\|f\|_{L^{p}(\lambda)},
\end{equation}
where $C$ does not depend on $f$.

Let us now verify the condition \eqref{crit_check1}. For this, we consider the cases: $r_{0}>1$ and $0<r_{0}\leq 1$. 

In the case $r_{0}>1$, we have $2<2r_{0}\leq |x|<2|z|$. For $r_{0}>1$, by
\eqref{A_alp} and \eqref{check2} one has
\begin{equation}\label{crit_check2}
\begin{split}
\int_{\{2r_{0}\leq |x|\}}\left(\sup_{\{|x|<2|z|<3|x|\}}|G^{c}_{d/p,\chi}(z)|\right)^{q}
&\frac{d\lambda(x)}{\left|\log\left({\rm e}+\frac{1}{|x|}\right)\right|^{r}|x|^{d}}\\&\leq
\int_{\{2r_{0}\leq |x|\}}\left(\sup_{\{|x|<2|z|<3|x|\}}|G^{c}_{d/p,\chi}(z)|\right)^{q}
\frac{d\lambda(x)}{|x|^{d}}\\&
\lesssim r_{0}^{-d}{\rm e}^{-qc^{\prime}\frac{r_{0}}{4}}.
\end{split}
\end{equation}
Then, as in \eqref{check4}, plugging \eqref{crit_check2} and \eqref{check3} into \eqref{crit_check1}, we obtain
\begin{multline}\label{crit_check3}
\left(\int_{\{r_{0}\leq |x|\}}\left(\sup_{\{|x|<2|z|<3|x|\}}|G^{c}_{d/p,\chi}(z)|\right)^{q}
\frac{d\lambda(x)}{|x|^{d}}\right)^{\frac{1}{q}}
\left(\int_{\{|x|< r_{0}\}}d\lambda(x)\right)^{\frac{1}{p^{\prime}}}\\\lesssim r_{0}^{-\frac{d}{q}}{\rm e}^{-c^{\prime}\frac{r_{0}}{4}}
{\rm e}^{C\frac{r_{0}}{p^{\prime}}}<\infty
\end{multline}
since $c$ (hence $c^{\prime}$) is large enough.

Let us now check \eqref{crit_check1} for $0<r_{0}\leq 1$. When $|z|>1$ using the exponential decay estimate of $G_{d/p,\chi}^{c}(z)$ from \eqref{A_alp} it is easy to verify \eqref{crit_check1}. So, let us show the case $|z|\leq 1$. In this case, taking into account $|x|<2|z|\leq 2$ we write
\begin{multline}\label{crit_check4}
\int_{\{2r_{0}\leq |x|\}}\left(\sup_{\{|x|<2|z|<3|x|\}}|G^{c}_{d/p,\chi}(z)|\right)^{q}
\frac{d\lambda(x)}{\left|\log\left({\rm e}+\frac{1}{|x|}\right)\right|^{r}|x|^{d}}
\\=
\int_{\{2r_{0}\leq |x|\leq 1\}}\left(\sup_{\{|x|<2|z|<3|x|\}}|G^{c}_{d/p,\chi}(z)|\right)^{q}\frac{d\lambda(x)}{\left|\log\left({\rm e}+\frac{1}{|x|}\right)\right|^{r}|x|^{d}}
\\+\int_{\{1<|x|<2\}}\left(
\sup_{\{|x|<2|z|<3|x|\}}|G^{c}_{d/p,\chi}(z)|\right)^{q}\frac{d\lambda(x)}{\left|\log\left({\rm e}+\frac{1}{|x|}\right)\right|^{r}|x|^{d}}.
\end{multline}
We see from \eqref{A_alp} that the second integral on the right-hand side of \eqref{crit_check4} is finite. For the first integral on the right-hand side, noting \eqref{A_alp} we deduce that
\begin{equation*}
\begin{split}
\int_{\{2r_{0}\leq |x|\leq 1\}}\left(\sup_{\{|x|<2|z|<3|x|\}}|G^{c}_{d/p,\chi}(z)|\right)^{q}&
\frac{d\lambda(x)}{\left|\log\left({\rm e}+\frac{1}{|x|}\right)\right|^{r}|x|^{d}}\\&
\leq \int_{\{2r_{0}\leq |x|\leq 1\}}\left( \sup_{\{|x|<2|z|<3|x|\}}|G^{c}_{d/p,\chi}(z)|\right)^{q}
\frac{d\lambda(x)}{|x|^{d}}\\&\lesssim
\int_{\{2r_{0}\leq |x|\leq 1\}}|x|^{-dq/p'-d}d\lambda(x)\\&
\lesssim r_{0}^{-dq/p'},
\end{split}
\end{equation*}
which implies with \eqref{crit_check4} that
\begin{equation*}
\begin{split}\left(\int_{\{2r_{0}\leq |x|\}}\left( \sup_{\{|x|<2|z|<3|x|\}}|G^{c}_{d/p,\chi}(z)|\right)^{q}
\frac{d\lambda(x)}{\left|\log\left({\rm e}+\frac{1}{|x|}\right)\right|^{r}|x|^{d}}\right)^{\frac{1}{q}}&
\left(\int_{\{|x|< 2r_{0}\}}d\lambda(x)\right)^{\frac{1}{p^{\prime}}}
\\&\leq C(r_{0}^{-d/p'}+1)r_{0}^{d/p'}\leq C
\end{split}
\end{equation*}
for any $0<r_{0}\leq 1$.

Now we estimate $N_{3}$. As in the case for $M_{3}$, we have for $N_{3}$ that
$$N_{3}\leq \int_{\G}\left(\int_{\{|y|\geq 2|x|\}}|f(y)|\left(
\sup_{\{|y|\leq 2|z|\leq 3|y|\}}|G_{d/p,\chi}^{c}(z)|\right)d\lambda(y)\right)^{q}\frac{d\lambda(x)}
{\left|\log\left({\rm e}+\frac{1}{|x|}\right)\right|^{r}|x|^{d}}.$$
For $N_{3}$ we apply the conjugate integral Hardy inequality \eqref{integ_Hardy2}, for which we need to check the following condition \eqref{integ_Hardy2_cond}: 
\begin{multline}\label{crit_check6}
\left(\int_{\{|x|\leq 2r_{0}\}}\frac{d\lambda(x)}{\left|\log\left({\rm e}+\frac{1}{|x|}\right)\right|^{r}|x|^{d}}\right)^{\frac{1}{q}}\\
\times\left(\int_{\{2r_{0}\leq |y|\}}\left(
\sup_{\{|y|\leq 2|z|\leq 3|y|\}}|G_{d/p,\chi}^{c}(z)|\right)^{p^{\prime}}d\lambda(y)\right)^{\frac{1}{p^{\prime}}}<\infty
\end{multline}
for all $r_{0}>0$. Indeed, once \eqref{crit_check6} has been established, the conjugate integral Hardy inequality \eqref{integ_Hardy2} yields 
\begin{equation}\label{N3}
N_{3}^{\frac{1}{q}}\leq C\|f\|_{L^{p}(\lambda)},
\end{equation}
where $C$ does not depend on $f$.

In order to check this, we consider the cases: $r_{0}>1$ and $0<r_{0}\leq 1$. If we write
\begin{multline*}
\int_{\{|x|\leq 2r_{0}\}}\frac{dx}{\left|\log\left({\rm e}+\frac{1}{|x|}\right)\right|^{r}|x|^{d}}
=\int_{\left\{|x|<\frac{1}{2}\right\}}\frac{dx}{\left|\log\left({\rm e}+\frac{1}{|x|}\right)\right|^{r}|x|^{d}}\\+
\int_{\left\{\frac{1}{2}\leqslant |x|\leq 2r_{0}\right\}}\frac{dx}{\left|\log\left({\rm e}+\frac{1}{|x|}\right)\right|^{r}|x|^{d}},
\end{multline*}
then we see that the first summand on the right-hand side of above is finite since $r>1$. For the second term, using
\eqref{check3} we have
\begin{equation}\label{crit_check7}
\int_{\left\{\frac{1}{2}\leq |x|\leq 2r_{0}\right\}}\frac{d\lambda(x)}{\left|\log\left({\rm e}+\frac{1}{|x|}\right)\right|^{r}|x|^{d}}
\leq \int_{\left\{\frac{1}{2}\leq |x|\leq 2r_{0} \right\}}\frac{d\lambda(x)}{|x|^{d}}\leq 2^{d}{\rm e}^{C_{5}r_{0}}
\end{equation}
for some positive constant $C_{5}$.
Combining \eqref{check10} and \eqref{crit_check7}, one obtains for $r_{0}>1$ that
\begin{multline*}
\left(\int_{\{|x|\leq 2r_{0}\}}\frac{d\lambda(x)}{\left|\log\left({\rm e}+\frac{1}{|x|}\right)\right|^{r}|x|^{d}}\right)^{\frac{1}{q}}
\left(\int_{\{2r_{0}\leq |y|\}}\left(\sup_{\{|y|\leq 2|z|\leq 3|y|\}}|G_{d/p,\chi}^{c}(z)|\right)^{p^{\prime}}d\lambda(y)\right)^{\frac{1}{p^{\prime}}}\\
\lesssim
(1+2^{d}{\rm e}^{C_{5}r_{0}})^{\frac{1}{q}}{\rm e}^{-c^{\prime}\frac{r_{0}}{4}}<\infty,
\end{multline*}
since $c$ (hence $c^{\prime}$) is large enough.

Now we check the condition \eqref{crit_check6} for $0<r_{0}\leq 1$. As above, for $|z|>1$ using \eqref{A_alp} it is straightforward to get \eqref{crit_check6}. So, for $|z|\leq 1$ we write
\begin{multline}\label{crit_check8}
\int_{\{2r_{0}\leq |y|\}}\left( \sup_{\{|y|\leq 2|z|\leq 3|y|\}}|G_{d/p,\chi}^{c}(z)|\right)^{p^{\prime}}d\lambda(y)\\
=
\int_{\{2r_{0}\leq |y|\leq 1\}}\left( \sup_{\{|y|\leq 2|z|\leq 3|y|\}}|G_{d/p,\chi}^{c}(z)|\right)^{p^{\prime}}d\lambda(y)\\+\int_{\{|y|>1\}}\left(
\sup_{\{|y|\leq 2|z|\leq 3|y|\}}|G_{d/p,\chi}^{c}(z)|\right)^{p^{\prime}}d\lambda(y).
\end{multline}
We note from \eqref{check10} that the second integral on the right-hand side of above is finite. Then, by \eqref{A_alp} we get
for the first integral that
\begin{equation*}
\int_{\{2r_{0}\leq |y|\leq 1\}}\left(\sup_{\{|y|\leq 2|z|\leq 3|y|\}}|G_{d/p,\chi}^{c}(z)|\right)^{p'}d\lambda(y)\leq C
\int_{\{2r_{0}\leq |y|\leq 1\}}|y|^{-d}d\lambda(y)
\leq C \log\left(\frac{1}{r_{0}}\right).
\end{equation*}
It follows with \eqref{crit_check8} that
\begin{equation}\label{crit_check9}
\int_{\{2r_{0}\leq |y|\}}\left(\sup_{\{|y|\leq 2|z|\leq 3|y|\}}|G_{d/p,\chi}^{c}(z)|\right)^{p^{\prime}}d\lambda(y)\leq
C\left(1+\log\left(\frac{1}{r_{0}}\right)\right).
\end{equation}
Since we have
$$\int_{\{ |x|\leq 2r_{0}\}}\frac{dx}{\left|\log\left({\rm e}+\frac{1}{|x|}\right)\right|^{r}|x|^{d}}\leqslant
C\left(\log\left({\rm e}+\frac{1}{r_{0}}\right)\right)^{-(r-1)},$$
and \eqref{crit_check9}, then taking into account $r>1$ and $q<(r-1)p^{\prime}$ we obtain that
\begin{equation}\label{crit_check10}
\begin{split}
\left(\int_{\{|x|\leq 2r_{0}\}}\frac{d\lambda(x)}{\left|\log\left({\rm e}+\frac{1}{|x|}\right)\right|^{r}|x|^{d}}\right)
^{\frac{1}{q}}&\left(\int_{\{2r_{0}\leq |y|\}} \left(\sup_{\{|y|\leq 2|z|\leq 3|y|\}}|G_{d/p,\chi}^{c}(z)|\right)^{p^{\prime}}d\lambda(y)\right)
^{\frac{1}{p^{\prime}}}\\&\leq C\left(\log \left({\rm e}+\frac{1}{r_{0}}\right)\right)^{-\frac{r-1}{q}}\left(1+\left(\log
\left(\frac{1}{r_{0}}\right)\right)^{\frac{1}{p^{\prime}}}\right)\\&
\leq C.
\end{split}
\end{equation}

Now it remains to estimate $N_{2}$. We rewrite $N_{2}$ as
\begin{multline*}
N_{2}=\\\sum_{k\in\mathbb{Z}}\int_{\{2^{k}\leqslant |x|<2^{k+1}\}}\left(\int_{\{|x|\leqslant 2|y|\leqslant 4|x|\}}|
G^{c}_{d/p}(y^{-1}x)f(y)|d\lambda(y)\right)^{q}\frac{d\lambda(x)}{\left|\log\left({\rm e}+\frac{1}{|x|}\right)\right|^{r}|x|^{d}}.
\end{multline*}
Since the function $\left(\log\left(\frac{1}{|x|}\right)\right)^{r}|x|^{d}$ is non-decreasing with respect to $|x|$ near the
origin, then we can say that there exists an integer $k_{0}\in\mathbb{Z}$ with $k_{0}\leqslant -3$ such that this function is non-decreasing in
$|x|\in(0,2^{k_{0}+1})$. We decompose $N_{2}$ with this $k_{0}$ as follows
\begin{equation}\label{N2}
N_{2}=N_{21}+N_{22},
\end{equation}
where
\begin{multline*}
N_{21}:=\\\sum_{k=-\infty}^{k_{0}}\int_{\{2^{k}\leqslant |x|<2^{k+1}\}}\left(\int_{\{|x|\leqslant 2|y|\leqslant 4|x|\}}|
G^{c}_{d/p}(y^{-1}x)f(y)|d\lambda(y)\right)^{q}\frac{d\lambda(x)}{\left|\log\left({\rm e}+\frac{1}{|x|}\right)\right|^{r}|x|^{d}}
\end{multline*}
and
\begin{multline*}N_{22}:=\\\sum_{k=k_{0}+1}^{\infty}\int_{\{2^{k}\leqslant |x|<2^{k+1}\}}\left(\int_{\{|x|\leqslant 2|y|\leqslant 4|x|\}}|
G^{c}_{d/p}(y^{-1}x)f(y)|d\lambda(y)\right)^{q}\frac{d\lambda(x)}{\left|\log\left({\rm e}+\frac{1}{|x|}\right)\right|^{r}|x|^{d}}.
\end{multline*}
Let us first estimate $N_{22}$. Using \eqref{M22} we obtain the following estimate for $N_{22}$
\begin{equation}\label{N22}
N_{22}\leqslant C \sum_{k=k_{0}+1}^{\infty}\int_{\{2^{k}\leqslant |x|<2^{k+1}\}}\left(\int_{\{|x|\leqslant 2|y|\leqslant
4|x|\}}| G^{c}_{d/p}(y^{-1}x)f(y)|dy\right)^{q}dx\leq C\|f \|^{q}_{L^{p}(\lambda)}.
\end{equation}
To complete the proof of Theorem \ref{main4_thm} it is left to estimate $N_{21}$. Note that the condition $|y|\leqslant 2|x|$ implies \begin{equation}\label{quasi_Euc_norm2}
3|x|=|x|+2|x|\geq |x|+|y|\geq |y^{-1}x|.
\end{equation}
Since $\left(\log\left(\frac{1}{|x|}\right)\right)^{r}|x|^{d}$ is non-decreasing in $|x| \in
(0,2^{k_{0}+1})$ and $3|x|\geqslant |y^{-1}x|$, we get $$\left(\log\left(\frac{1}{|x|}\right)\right)^{r}|x|^{d}\geq
\left(\log\left(\frac{1}{\left|\frac{y^{-1}x}{3}\right|}\right)\right)^{r}\left|\frac{y^{-1}x}{3}\right|^{d}.$$
Then, these and \eqref{A_alp} give
\begin{multline*}
N_{21}\leq \\ C\sum_{k=-\infty}^{k_{0}}\int_{\{2^{k}\leq|x|<2^{k+1}\}}\left(\int_{\{|x|\leq2|y|\leq4|x|\}}
|y^{-1}x|^{-\frac{d}{p^{\prime}}}|f(y)|d\lambda(y)\right)^{q}\frac{d\lambda(x)}
{\left(\log\left(\frac{1}{|x|}\right)\right)^{r}|x|^{d}}\\= C\sum_{k=-\infty}^{k_{0}}\int_{\{2^{k}\leq|x|<2^{k+1}\}}\left(\int_{\{|x|\leq2|y|\leq4|x|\}}\frac{|y^{-1}x|
^{-\frac{d}{p^{\prime}}}|f(y)|}{\left(\left(\log\left(\frac{1}{|x|}\right)\right)^{r}|x|^{d}\right)^{\frac{1}{q}}}d\lambda(y)
\right)^{q}d\lambda(x)\\
\leq C\sum_{k=-\infty}^{k_{0}}\int_{\{2^{k}\leq|x|<2^{k+1}\}}\left(\int_{\{|x|\leq2|y|\leq4|x|\}}\frac{|y^{-1}x|
^{-\frac{d}{p^{\prime}}}|f(y)|d\lambda(y)}{\left(\left(\log\left(\frac{3}{|y^{-1}x|}\right)\right)^{r}
\left|\frac{y^{-1}x}{3}\right|^{d}\right)^{\frac{1}{q}}}\right)^{q}d\lambda(x).
\end{multline*}
Since the conditions $|x|\leq2|y|\leq4|x|$ and $2^{k}\leq|x|<2^{k+1}$ with $k\leq k_{0}$ imply $2^{k-1}\leq|y|<2^{k+2}$, while \eqref{quasi_Euc_norm2} and $k_{0}\leq-3$ yield $|y^{-1}x|\leq3|x|<3\cdot 2^{k_{0}+1}\leq3/4$. By these and
setting $$g_{2}(x):=\frac{\chi_{B_{\frac{3}{4}}(0)}(x)}{\left(\log\left(\frac{1}{|x|}\right)\right)
^{\frac{r}{q}}|x|^{\frac{d}{q}+\frac{d}{p'}}},$$
we obtain for $N_{21}$ that
\begin{multline*}
N_{21}\\ \leq C\sum_{k=-\infty}^{k_{0}}\int_{\{2^{k}\leq|x|<2^{k+1}\}}\left(\int_{\{|x|\leq2|y|\leq4|x|\}}\frac{|f(y)|d\lambda(y)}{
{\left(\log\left(\frac{1}{|y^{-1}x|}\right)\right)^{\frac{r}{q}}|y^{-1}x|^{\frac{d}{q}+\frac{d}{p'}}}}\right)^{q}
d\lambda(x)\\\leq C\sum_{k=-\infty}^{k_{0}}\|[f\cdot\chi_{\{2^{k-1}\leq |\cdot|<2^{k+2}\}}]\ast g\|^{q}_{L^{q}(\lambda)}.
\end{multline*}
Since $p\leq q<(r-1)p'$, we apply Young's inequality \eqref{Young_for} for $1+\frac{1}{q}=\frac{1}{\tilde{r}}+\frac{1}{p}$ with
$\tilde{r}\in[1,\infty)$ to get
\begin{equation}\label{N21}
N_{21}\leq C
\|g_{2}\|^{q}_{L^{\tilde{r}}(\lambda)}\sum_{k=-\infty}^{k_{0}}\|f\cdot\chi_{\{2^{k-1}\leq|\cdot|<2^{k+2}\}}\|^{q}_{L^{p}(\lambda)}\leq
C\|f\|^{q}_{L^{p}(\lambda)},
\end{equation}
provided that $g_{2}\in L^{\tilde{r}}(\lambda)$. Since $\left(\frac{d}{q}+\frac{d}{p'}\right)\tilde{r}=d$,
$\frac{r\tilde{r}}{q}=\frac{rp'}{p'+q}$ and $q<(r-1)p'$, then the change of the variable $t=\log\left(\frac{1}{|x|}\right)$ gives
$$\|g_{2}\|^{\tilde{r}}_{L^{\tilde{r}}(\lambda)}=
\int_{B(0,3/4)}\frac{d\lambda(x)}{\left(\log\left(\frac{1}{|x|}\right)\right)^{\frac{rp'}{p'+q}}|x|^{d}}
\leq C\int_{\log\left(\frac{4}{3}\right)}^{\infty}\frac{dt}{t^{\frac{rp'}{p'+q}}}<\infty.$$
Thus, \eqref{N1}, \eqref{N3}, \eqref{N2}, \eqref{N22}, \eqref{N21} and \eqref{N123} complete the proof of Theorem
\ref{main4_thm}.
\end{proof}
\subsection{Proof of Theorems \ref{CKN_thm}, \ref{CKN_thm_crit}, \ref{HLS_thm} and Corollary \ref{uncer_cor}}

First, let us prove Theorem \ref{CKN_thm}, using the Hardy-Sobolev-Rellich inequality \eqref{main1_thm_for1}.
\begin{proof}[Proof of Theorem \ref{CKN_thm}]
Since $\theta>(r-q)/r$, using H\"{o}lder's inequality for $\frac{q-(1-\theta)r}{q}+\frac{(1-\theta)r}{q}=1$, we calculate
\begin{multline*}
\||x|^{a}f\|_{L^{r}(\mu_{\chi^{\widetilde{q}/p}\delta^{1-\widetilde{q}/p}})}= \left(\int_{\G}\frac{|f(x)|^{\theta r}}{|x|^{r(b(1-\theta)-a)}}\cdot\frac{|f(x)|^{(1-\theta)r}}{|x|^{-br(1-\theta)}}d\mu_{\chi^{\widetilde{q}/p}\delta^{1-\widetilde{q}/p}}(x)
\right)^{\frac{1}{r}}\\\leq
\left(\left(\int_{\G}\frac{|f(x)|^{\theta r\frac{q}{q-(1-\theta)r}}}{|x|^{r(b(1-\theta)-a)\frac{q}{q-(1-\theta)r}}}d\mu_{\chi^{\widetilde{q}/p}\delta^{1-\widetilde{q}/p}}(x)\right)
^{\frac{q-(1-\theta)r}{q}}\right.\\\left.
\times
\left(\int_{\G}\frac{|f(x)|^{(1-\theta)r\frac{q}{(1-\theta)r}}}{|x|^{-br(1-\theta)\frac{q}{(1-\theta)r}}}
d\mu_{\chi^{\widetilde{q}/p}\delta^{1-\widetilde{q}/p}}(x)\right)^{\frac{(1-\theta)r}{q}}\right)^{\frac{1}{r}}
\end{multline*}
$$
=\left\|\frac{f}{|x|^{\frac{b(1-\theta)-a}{\theta}}}\right\|^{\theta}
_{L^{\frac{qr\theta}{q-(1-\theta)r}(\mu_{\chi^{\widetilde{q}/p}\delta^{1-\widetilde{q}/p}})}}
\||x|^{b}f\|^{1-\theta}_{L^{q}(\mu_{\chi^{\widetilde{q}/p}\delta^{1-\widetilde{q}/p}})}.
$$
Now, since we have $\alpha>0$, $p\leq q\theta r/(q-(1-\theta)r)$, $0\leq qr(b(1-\theta)-a)/(q-(1-\theta)r)<d$ and $1/p-(q-(1-\theta)r)/(qr\theta)\leq \alpha/d-(b(1-\theta)-a)/(\theta d)$, then applying \eqref{main1_thm_for1} we obtain \eqref{CKN}.
\end{proof}
Similarly, one can obtain Theorem \ref{CKN_thm_crit} from Theorem \ref{main3_thm}.

Now let us give the proof of Hardy-Littlewood-Sobolev inequality \eqref{HLS_thm_ineq} on general Lie group:
\begin{proof}[Proof of Theorem \ref{HLS_thm}] Using H\"{o}lder's inequality for $q/(p+q)+p/(p+q)=1$ one has
\begin{equation}\label{HLS_proof_eq1}
\begin{split}\left|\int_{\G}\int_{\G}\right.&\left.\frac{\overline{f(x)}g(y)G_{a_{2},\chi}^{c}(y^{-1}x)}{|x|^{a_{1}}|y|^{\beta}}
d\mu_{\chi^{(p+q)/pq}\delta^{1-(p+q)/pq}}(x)d\rho(y)\right|\\&=
\left|\int_{\G}\overline{f(x)}\frac{\left(\frac{g}{|x|^{\beta}}\ast G_{a_{2},\chi}^{c}\right)(x)}{|x|^{a_{1}}}
d\mu_{\chi^{(p+q)/pq}\delta^{1-(p+q)/pq}}(x)\right|\\&\leq
\|f\|_{L^{(p+q)/q}(\mu_{\chi^{(p+q)/pq}\delta^{1-(p+q)/pq}})}\left\|\frac{\frac{g}{|x|^{\beta}}\ast G_{a_{2},\chi}^{c}}{|x|^{a_{1}}}\right\|_{L^{(p+q)/p}(\mu_{\chi^{(p+q)/pq}\delta^{1-(p+q)/pq}})}.
\end{split}
\end{equation}
Since $\alpha\geq 0$, $1/p-q/(p+q)\leq \alpha/d$, and the fact that $0\leq 1/p-q/(p+q)$ implies $(p+q)/q\geq p$, then applying unweighted version of the Hardy-Sobolev-Rellich inequality \eqref{main1_thm_for1} we have
\begin{equation}\label{HLS_proof_eq2}
\|f\|_{L^{(p+q)/q}(\mu_{\chi^{(p+q)/pq}\delta^{1-(p+q)/pq}})}\lesssim \|f\|_{L^{p}_{\alpha}(\mu_{\chi})}.
\end{equation}
Since $0\leq a_{1}<dp/(p+q)$, $a_{2}>0$ and $1/q-p/(p+q)\leq (a_{2}-a_{1})/d$, then \eqref{main1_thm_for1} implies
\begin{equation}\label{HLS_proof_eq3}
\left\|\frac{\frac{g}{|x|^{\beta}}\ast G_{a_{2},\chi}^{c}}{|x|^{a_{1}}}\right\|_{L^{(p+q)/p}(\mu_{\chi^{(p+q)/pq}\delta^{1-(p+q)/pq}})}
\lesssim \left\|\frac{g}{|x|^{\beta}}\right\|_{L^{q}(\mu_{\chi})}\lesssim \|g\|_{L^{q}(\mu_{\chi})},
\end{equation}
where we have used \eqref{lef_rig_Hardy1} in the last inequality since $0\leq \beta<d/q$.
Thus, plugging \eqref{HLS_proof_eq2} and \eqref{HLS_proof_eq3} into \eqref{HLS_proof_eq1} we obtain \eqref{HLS_thm_ineq}.
\end{proof}

Now we prove Corollary \ref{uncer_cor}.
\begin{proof}[Proof of Corollary \ref{uncer_cor}] By \eqref{main1_thm_for1} and H\"{o}lder's inequality for $1/q+1/q'=1$, we obtain
\begin{multline*}
\|f\|_{L^{p}_{\alpha}(\mu_{\chi})}\||x|^{\frac{\beta}{q}}f\|_{L^{q^{\prime}}(\mu_{\chi^{q/p}\delta^{1-q/p}})}
\\ \gtrsim \left\|\frac{f}{|x|^{\frac{\beta}{q}}}\right\|_{L^{q}(\mu_{\chi^{q/p}\delta^{1-q/p}})}
\||x|^{\frac{\beta}{q}}f\|_{L^{q^{\prime}}(\mu_{\chi^{q/p}\delta^{1-q/p}})}\\
\geq \|f\|^{2}_{L^{2}(\mu_{\chi^{q/p}\delta^{1-q/p}})},
\end{multline*}
which is \eqref{uncer_1}.

Similarly, Theorem \ref{main3_thm} implies the second part of Corollary \ref{uncer_cor}.
\end{proof}

\section{Appendix: The case of compact Lie groups}
\label{SEC:app_compact}
In this section we show that the obtained results on noncompact Lie groups actually hold also on compact Lie groups in a similar way. In the setting of compact Lie groups, we have $\delta=1$ hence $d\lambda =d\rho$, and the continuous positive character $\chi$ must be identically equal to $1$. We refer to \cite{RT10} for the background material as well as the Fourier analysis on compact Lie groups.

Let us recall the following result:
\begin{thm}\cite[VIII.2.9 Theorem]{VCS92}\label{est_heat_com_great1}
If $\G$ has a polynomial growth, there exist two positive constants $C_{1}$ and $C_{2}$ such that
\begin{equation}\label{est_heat_com_for_great1}
C_{1}V(\sqrt{t})^{-1}\exp(C_{2}|x|^{2}/t)\leq p_{t}(x)\leq C_{2}V(\sqrt{t})^{-1}\exp(-C_{1}|x|^{2}/t)
\end{equation}
for all $t>0$ and $x\in \G$.
\end{thm}
Now we give an analogue of Lemma \ref{lemma4.1_orig} on compact Lie groups when $0<\alpha<d$, since we have actually used only this case of Lemma \ref{lemma4.1_orig} in the noncompact case:
\begin{lem}\label{est_G_com_lem} Let $0<\alpha<d$. If $c>0$ is sufficiently large, then we have
\begin{equation}\label{less1_com}
|G_{\alpha,\chi}^{c}|\leq C|x|^{\alpha-d}
\end{equation}
for all $x\in\G$ and some positive constant $C$.
\end{lem}
\begin{proof}[Proof of Lemma \ref{est_G_com_lem}] Taking into account Theorem \ref{est_heat_com_great1} with \eqref{V_less1} and \eqref{V_great1} as well as the relation \eqref{G_bessel_rep}, we have
\begin{equation*}
\begin{split}
|G_{\alpha,\chi}^{c}(x)|&=\left|C(\alpha)\int_{0}^{\infty}t^{\alpha/2-1}{\rm e}^{-ct}p_{t}(x)dt\right|\\&
\lesssim \int_{0}^{1}t^{(\alpha-d)/2-1}{\rm e}^{-ct}{\rm e}^{-C|x|^{2}/t}dt\\
&+\int_{1}^{\infty}t^{(\alpha-D)/2-1}{\rm e}^{-ct}{\rm e}^{-C|x|^{2}/t}dt=:G_{1}(x)+G_{2}(x).
\end{split}
\end{equation*}
It is easy to see that $G_{2}(x)\lesssim1$, since $c$ is large enough.

In the exact same way as in the proof of Lemma \ref{lemma4.1_orig} (see \cite[Proof of Lemma 4.1]{BPTV18}), using the change of variables $|x|^{2}/t=u$ we arrive at
\begin{equation*}
\begin{split}
G_{1}(x)&\lesssim \int_{0}^{1}t^{(\alpha-d)/2-1}{\rm e}^{-C|x|^{2}/t}dt
=|x|^{\alpha-d}\int_{|x|^{2}}^{\infty}u^{\frac{d-\alpha}{2}}{\rm e}^{-Cu}\frac{du}{u},
\end{split}
\end{equation*}
which gives the estimate \eqref{less1_com}.
\end{proof}
Since we have Lemma \ref{est_G_com_lem}, $\underset{x\in B_{r}}{{\rm sup}}\chi(x)={\rm const}$ and $\int_{B_{r}}\chi d\rho= {\rm const}$ for every $r\gg1$, which play key roles in the proof of Theorems \ref{main1_thm} and \ref{main3_thm}, then we also have these Theorems \ref{main1_thm} and \ref{main3_thm} with $\delta=1$ on compact Lie group.

Note that in the proof of Theorems \ref{CKN_thm}, \ref{CKN_thm_crit} and \ref{HLS_thm}, and Corollary \ref{uncer_cor}, we only use H\"{o}lder's inequality, and Theorems \ref{main1_thm} and \ref{main3_thm}. Therefore, since now we have Theorems \ref{main1_thm} and \ref{main3_thm} with $\delta=1$ on compact Lie groups, then Theorems \ref{CKN_thm}, \ref{CKN_thm_crit} and \ref{HLS_thm}, and Corollary \ref{uncer_cor} also hold on compact Lie group, with $\delta=1$.

\medskip
\paragraph{{\bf Acknowledgments}} The authors would like to thank the anonymous referees for
their helpful and constructive comments that greatly contributed to improving the final
version of the paper. The authors would also like to thank Zhirayr Avetisyan for discussions.

\end{document}